\documentclass[10pt,reqno]{amsart} 
\usepackage[a4paper, margin=2.3cm]{geometry}

\usepackage{amsthm}

\usepackage{bbm}
\usepackage{mathrsfs}

\usepackage[utf8]{inputenc}
\usepackage{xypic}

\usepackage{enumitem}

\usepackage{float}
\usepackage{amsbsy}
\usepackage[all]{xy}\usepackage{xypic}
\usepackage{braket}
\usepackage[colorlinks = true,citecolor=black]{hyperref}

\hypersetup{
     colorlinks=true,
     linkcolor=blue,
     filecolor=blue,
     citecolor = blue,      
     urlcolor=cyan,
     }

\newtheorem{theorem}{Theorem}[section]
\newtheorem*{theorem*}{Theorem}
\newtheorem{theorem-non}{Theorem}
\newtheorem{lemma-non}{Lemma}

\theoremstyle{definition} 
\newtheorem{thm}{Theorem}

\theoremstyle{definition} 
\newtheorem{corollarynon}{Corollary}

\newtheorem{conjecture-non}{Conjecture}

\newtheorem{corollary-non}{Corollary}
\newtheorem{proposition}[theorem]{Proposition}
\newtheorem{lemma}[theorem]{Lemma}
\newtheorem*{lemma*}{Lemma}
\newtheorem{corollary}[theorem]{Corollary}

\newtheorem*{conjecture*}{Conjecture}

\theoremstyle{definition}
\newtheorem{definition}[theorem]{Definition}

\theoremstyle{remark}
\newtheorem{remark}[theorem]{Remark}

\numberwithin{equation}{section}

\begin{document}
\title[K\"{a}hler-Ricci flow on rational homogeneous varieties]{K\"{a}hler-Ricci flow on rational homogeneous varieties}

\author{Eder M. Correa}


\address{{IMECC-Unicamp, Departamento de Matem\'{a}tica. Rua S\'{e}rgio Buarque de Holanda 651, Cidade Universit\'{a}ria Zeferino Vaz. 13083-859, Campinas-SP, Brazil}}
\address{E-mail: {\rm ederc@unicamp.br.}}

\begin{abstract} 
In this work, we study the K\"{a}hler-Ricci flow on rational homogeneous varieties exploring the interplay between projective algebraic geometry and representation theory which underlies the classical Borel-Weil theorem. By using elements of representation theory of semisimple Lie groups and Lie algebras, we give an explicit description for all solutions of the K\"{a}hler-Ricci flow with homogeneous initial condition. This description enables us to compute explicitly the maximal existence time for any solution starting at a homogeneous K\"{a}hler metric and obtain explicit upper and lower bounds for several geometric quantities along the flow, including curvatures, volume, diameter, and the first non-zero eigenvalue of the Laplacian. As an application of our main result, we investigate the relationship between certain numerical invariants associated to ample divisors and numerical invariants arising from solutions of the K\"{a}hler-Ricci flow in the homogeneous setting. 
\end{abstract}

\maketitle

\hypersetup{linkcolor=black}
\tableofcontents

\hypersetup{linkcolor=black}
\section{Introduction}
\subsection{Motivations} Given a compact K\"{a}hler manifold $(X,\omega_{0})$ of complex dimension $n$, a solution of the K\"{a}hler-Ricci flow on $X$ starting at $\omega_{0}$ is a family of K\"{a}hler metrics $\omega(t)$ solving 
\begin{equation}
\label{KRF}
\displaystyle{\frac{\partial}{\partial t}\omega(t) = -{\rm{Ric}}(\omega(t))}, \ \ \omega(0) = \omega_{0}.
\end{equation}
From the short-time existence result of Hamilton \cite{Hamilton} (see also \cite{DeTurck}), and the fact that a maximal solution to the Ricci flow preserves the K\"{a}hler condition (e.g. \cite{Hamilton1}), it follows that the initial-value problem (\ref{KRF}) always admits a unique solution $\omega(t)$ defined on a maximal interval $[0,T)$, with $0<T \leq \infty$. Moreover, a result of Tian and Zhang \cite{TianZhang} gives a concrete characterization for the maximal existence time $T$. It is well-known (e.g. \cite{Tsuji2}, \cite{TianZhang}) that the flow (\ref{KRF}) has a global solution (i.e. $T = \infty$) if and only if the canonical line bundle $K_{X}$ of $X$ is nef or equivalently, if and only if $X$ is a minimal model \cite{MModel}, \cite{Cascini}. On the other hand, if $T < \infty$, we say that the flow (\ref{KRF}) has a finite time singularity at $T$. In this last case, the limiting class of the flow $[\omega_{T}] = [\omega_{0}] - Tc_{1}(X)$, which is nef but not K\"{a}hler, encodes the behavior of the singularity formation set of the flow (\ref{KRF}), see for instance \cite{Feldman}, \cite{Zhang}, \cite{Collins}, and references therein. 

In the particular setting of finite time singularity ($T < \infty$), from \cite{Hamilton} we have that the norm of the Riemann curvature tensor is unbounded on $X \times [0,T)$. Also, it was shown in \cite{Sesum} that the norm of the Ricci tensor has to become unbounded as $t \nearrow T$. Further, it was proved in \cite{Zhang} that the scalar curvature also becomes unbounded for finite time singularity. In \cite{SesumTian}, following Perelman's idea, Sesum and Tian proved that, if $c_{1}(X) > 0$ and $\omega_{0} \in c_{1}(X)$, then
\begin{equation}
\label{conj1}
R(t) \leq \frac{C}{T - t},
\end{equation}
where $R(t) = R(\omega(t))$ is the scalar curvature of $\omega(t)$ and $C$ is a uniform constant. In \cite{Zhang2}, it was shown in a quite general setting, that $R(t) \leq C/(T - t)^{2}$. More generally, we say that $\omega(t)$ is a Type-I solution of (\ref{KRF}) if 
\begin{equation}
\label{conj2}
|{\rm{Rm}}| \leq \frac{C}{T - t},
\end{equation}
for some uniform constant $C$, see for instance \cite{EndersMullerTopping}. In the above setting, there was a folklore speculation that all finite time singularities along the K\"{a}hler-Ricci flow are of Type-I, e.g. \cite{SongWeinkove}. However, by the recent work on the compactification spaces of reductive Lie groups by Li-Tian-Zhu, see \cite{LiTianZhu}, we have that this folklore speculation does not hold. These results are related to Hamilton-Tian's conjecture \cite{HTconjecture}, which was recently proved (independently) in \cite{ChenWang}, \cite{Bamler}, and \cite{WangZhu}. Besides the study of curvature bounds, the understanding of the evolution of other basic geometric quantities (such as volumes, diameters, etc.) also has been a basic task in the study of the K\"{a}hler-Ricci flow. Diameter bounds for solutions of the K\"{a}hler-Ricci flow as we approach a singularity are not easy to get. In general, it is expected the following \cite{Tosatti}:

\begin{conjecture-non}
\label{conj4}
Let $\omega = \omega(t)$ be a solution of the K\"{a}hler-Ricci flow (\ref{KRF}) on the maximal time interval $[0,T)$. If $T < \infty$, then
\begin{equation}
{\rm{diam}}(X,\omega(t)) \leq C,
\end{equation}
for all $t \in [0,T)$.
\end{conjecture-non}
This conjecture is known when $X$ is Fano and $\omega_{0} \in \lambda c_{1}(X)$, for some $\lambda > 0$, see for instance \cite{SesumTian}. The above conjecture is also known in the case when the limiting class $[\omega_{T}] = [\omega_{0}] - Tc_{1}(X)$ is equal to $\pi^{\ast}(\omega_{Y})$, where $\pi \colon X \to Y$ is the blowup of a compact K\"{a}hler manifold $Y$ at finitely many distinct points and $\omega_{Y}$ is a K\"{a}hler metric on $Y$ (e.g. \cite{SongWeinkove1}), and it is also proved in \cite{SongSzekelyhidiWeinkove} for some special Fano fibrations. Further results on diameter bounds can be found in \cite{IlmanenKnopf}, \cite{Topping}, \cite{ZhangQi}. Inspired by the above facts and Conjecture \ref{conj4}, in this paper we study the K\"{a}hler-Ricci flow on rational homogeneous varieties. As it was shown in \cite{Hamilton} (see also \cite{Kotschwar}), the Ricci flow preserves the isometries of the initial Riemannian manifold. Thus, if the initial metric $\omega_{0}$ in (\ref{KRF}) is homogeneous, we have that the evolving metric remains homogeneous during the flow. A solution of the K\"{a}hler-Ricci flow is homogeneous if it is homogeneous at any time. The Ricci flow on homogeneous Riemannian manifolds has been investigated by many authors, e.g. \cite{AnastassiouChrysikos}, \cite{IsenberJackson}, \cite{BohmWilking}, \cite{Lauret}, \cite{Optimal}, \cite{BohmLafuente}, \cite{Lauret1}, \cite{GramaMartins}, see also \cite{Lauret2} and references therein. However, there are very few results on K\"{a}hler-Ricci flow on homogeneous K\"{a}hler manifolds (unless they are viewed as homogeneous Riemannian manifolds). In general, bounds for geometric quantities that are sharp for Riemannian manifolds are not sharp for K\"{a}hler manifolds\footnote{A well-known example of this fact is given by the lower bound of the first non-zero eigenvalue of the Laplacian of closed Riemannian manifolds and closed K\"{a}hler manifolds, see for instance \cite{Andre}.}. Thus, it seems suitable to investigate the K\"{a}hler-Ricci flow on homogeneous manifolds taking into account tools from K\"{a}hler geometry which are not available in the Riemannian geometry setting. With this idea in mind, the aim of this paper is to study the K\"{a}hler-Ricci flow on rational homogeneous varieties exploring the interplay between projective algebraic geometry and representation theory which underlies the classical Borel-Weil theorem. By using elements of representation theory of semisimple Lie groups and Lie algebras, in the setting of rational homogeneous varieties, we give an explicit description for all solutions of the K\"{a}hler-Ricci flow with homogeneous initial condition. This description enables us to obtain explicit upper and lower bounds for several geometric quantities along the flow, including curvatures, volume, diameter, and the first non-zero eigenvalue of the Laplacian. In particular, we prove that Conjecture \ref{conj4} holds for any solution of the K\"{a}hler-Ricci flow starting at any homogeneous K\"{a}hler metric. In the homogeneous setting, these results generalize some results provided in \cite{SesumTian} on diameter and curvature bounds under the hypothesis that $\omega_{0} \in c_{1}(X)$. Also, as an application of our main result, we investigate the relationship between numerical invariants associated to ample divisors and numerical invariants arising from homogeneous solutions of the K\"{a}hler-Ricci flow.

\subsection{Main results} A rational homogeneous variety can be described as a quotient $X_{P} = G^{\mathbbm{C}}/P$, where $G^{\mathbbm{C}}$ is a semisimple complex algebraic group and $P$ is a parabolic subgroup (Borel-Remmert \cite{BorelRemmert}). Regarding $G^{\mathbbm{C}}$ as a complex analytic space, without loss of generality, we may assume that $G^{\mathbbm{C}}$ is a connected simply connected complex simple Lie group. Fixed a compact real form $G \subset G^{\mathbbm{C}}$, and considering $X_{P} = G/G \cap P$ as a $G$-space, in this paper we are interested in the homogeneous solutions of K\"{a}hler-Ricci flow (\ref{KRF}) on $X_{P}$. In the setting of rational homogeneous varieties we have a good description for the cohomology information underlying the K\"{a}hler-Ricci flow with homogeneous initial condition in terms of Lie theory, and it allows us to solve the parabolic PDE provided by (\ref{KRF}) just working out at the cohomology level. In fact, a solution of the K\"{a}hler-Ricci flow on $X_{P}$ defines a curve in the K\"{a}hler cone $\mathcal{K}_{X_{P}} \subset H^{1,1}(X_{P},\mathbbm{R})$, and since every $G$-invariant K\"{a}hler metric has the same Ricci form (e.g. \cite{MATSUSHIMA}), any homogeneous solution of the K\"{a}hler-Ricci flow (\ref{KRF}) satisfies ${\rm{Ric}}(\omega(t)) = {\rm{Ric}}(\omega_{0})$, $\forall t \in [0,T)$. From the uniqueness of $G$-invariant representatives in each cohomology class, the problem of solving the K\"{a}hler-Ricci flow on $X_{P}$ with a homogeneous initial condition reduces to the problem of solving the ODE defined by the tangent vector $-2\pi c_{1}(X_{P}) \in T_{[\omega_{0}]}\mathcal{K}_{X_{P}} = H^{1,1}(X_{P},\mathbbm{R})$. The solution of the K\"{a}hler-Ricci flow obtained from this ODE is given by
\begin{equation}
\label{generalsolution}
\omega(t) = \omega_{0} - t{\rm{Ric}}(\omega_{0}), \ \ t \in [0,T).
\end{equation}
In particular, notice that it also shows that every homogeneous solution of the K\"{a}hler-Ricci flow gives rise to a homogeneous solution of the continuity equation \cite{LaNaveTian} and vice-versa. The open convex cone $\mathcal{K}_{X_{P}}$ can be described in terms of the generators of the character group of $P \subset G^{\mathbbm{C}}$. More precisely, under the isomorphism
\begin{equation}
\label{characterintegral}
{\text{Hom}}(P,\mathbbm{C}^{\times}) \cong H^{1,1}(X_{P},\mathbbm{Z}),
\end{equation}
see for instance \cite{Popov}, the Chern classes of the line bundles associated to the generators of ${\text{Hom}}(P,\mathbbm{C}^{\times})$ define a suitable integral basis for the vector space $H^{1,1}(X_{P},\mathbbm{R})$, and they also span the convex cone $\mathcal{K}_{X_{P}}$. Based on these facts, the purpose of our main theorem is to use the isomorphism (\ref{characterintegral}) in order to obtain a concrete description for the homogeneous solutions (\ref{generalsolution}), as well as their maximal existence time, by means of the machinery of representation theory underlying the classical Borel-Weil theorem. In this way, we prove the following:

\begin{thm}
\label{Theo1}
Let $\omega_{0}$ be a $G$-invariant K\"{a}hler metric on a rational homogeneous variety $X_{P}$. Then the unique smooth solution $\omega(t)$ of the K\"{a}hler-Ricci flow on $X_{P}$ starting at $\omega_{0}$ satisfies the following:
\begin{enumerate}
\item[1)] $\omega(t)$ can be described locally in the explicit form
\begin{equation}
\omega(t) = \sum_{\alpha \in \Sigma \backslash \Theta}\bigg [ \int_{\mathbbm{P}_{\alpha}^{1}}\frac{\omega_{0}}{2\pi}- t\langle \delta_{P}, h_{\alpha}^{\vee} \rangle \bigg ]\sqrt{-1} \partial \overline{\partial}\log \big (||s_{U}v_{\varpi_{\alpha}}^{+}||^{2}\big ), \ \ \forall t \in [0,T),
\end{equation}
\end{enumerate}
for some local section $s_{U} \colon U \subset X_{P} \to G^{\mathbbm{C}}$, where $\mathbbm{P}_{\alpha}^{1} \subset X_{P}$, $\alpha \in \Sigma \backslash \Theta$, are generators of ${\rm{NE}}(X_{P})$;
\begin{enumerate}
\item[2)] The maximal existence time $T = T(\omega_{0})$ of $\omega(t)$ is given explicitly by
\begin{equation}
T(\omega_{0}) = \min_{\alpha \in \Sigma \backslash \Theta}  \int_{\mathbbm{P}_{\alpha}^{1}}\frac{\omega_{0}}{2\pi \langle \delta_{P}, h_{\alpha}^{\vee} \rangle} ;
\end{equation}

\item[3)] The scalar curvature $R(t)$ of $\omega(t)$ has the following explicit form
\begin{equation}
R(t) = -\sum_{\beta \in \Pi^{+} \backslash \langle \Theta \rangle^{+}}\frac{d}{dt}\log \bigg \{ \sum_{\alpha \in \Sigma \backslash \Theta} \bigg [ \int_{\mathbbm{P}_{\alpha}^{1}}\frac{\omega_{0}}{2\pi}- t\langle \delta_{P}, h_{\alpha}^{\vee} \rangle \bigg ] \langle \varpi_{\alpha}, h_{\beta}^{\vee} \rangle\bigg \}, \ \ \forall t \in [0,T);
\end{equation}

\item[4)] For all $0 \leq t < T$ we have
\begin{equation}
\frac{1}{\sqrt{n}(T-t)}\leq \frac{1}{\sqrt{n}} R(t) \leq |{\rm{Ric}}| \leq R(t) \leq \frac{n}{T - t}, \ \ and \ \ |{\rm{Rm}}| \leq \frac{C(n)}{T - t},
\end{equation}
\end{enumerate}
where $C(n)$ is a uniform constant which depends only on $n = \dim_{\mathbbm{C}}(X_{P})$;
\begin{enumerate}
\item[5)] For all $0 \leq t < T$ we have
\begin{equation}
\bigg [1-\frac{t}{T} \bigg]^{n}{\rm{Vol}}(X_{P},\omega_{0}) \leq {\rm{Vol}}(X_{P},\omega(t)) \leq \bigg [1-\frac{t}{T} \bigg] {\rm{Vol}}(X_{P},\omega_{0});
\end{equation}
\item[6)] For all $0 \leq t < T$ we have ${\rm{Ric}}(\omega(t)) \geq \frac{1}{C(\omega_{0})}$, such that 
\begin{equation}
\label{lowerboundricci}
C(\omega_{0}) =  \max_{\alpha \in \Sigma \backslash \Theta}  \int_{\mathbbm{P}_{\alpha}^{1}}\frac{\omega_{0}}{\pi \langle \delta_{P}, h_{\alpha}^{\vee} \rangle} .
\end{equation}
\end{enumerate}
In particular, for all $0 \leq t < T$, it follows that
\begin{equation}
{\rm{diam}}(X_{P},\omega(t)) \leq \pi \sqrt{(2n-1)C(\omega_{0})} \ \ \ \ \ \  {\text{and}} \ \ \ \ \ \ \ \frac{2}{C(\omega_{0})} \leq \lambda_{1}(t) \leq 2R(t) \Bigg [ \prod_{\alpha \succ 0} \frac{\langle \varrho^{+} + \delta_{P},h_{\alpha} \rangle}{\langle \delta_{P}, h_{\alpha} \rangle}\Bigg ],
\end{equation}
where $\lambda_{1}(t) = \lambda_{1}(X_{P},\omega(t))$ is the first non-zero eigenvalue of the Laplacian $\Delta_{\omega(t)} = {\rm{div} \circ {\rm{grad}}}$, $\forall t \in [0,T)$.
\end{thm}
The result above provides an explicit description for the unique solution of the K\"{a}hler-Ricci flow associated to any (homogeneous) initial data $(X_{P},\omega_{0})$ purely in terms of Lie theory. Actually, following the ideas of \cite{AZAD}, \cite{CorreaGrama}, \cite{Correa}, one can compute explicitly any solution as described in item (1) using algebraic tools of representation theory of complex semisimple Lie algebras. Particularly, from item $(6)$ of Theorem \ref{Theo1}, we have the following:
\begin{corollarynon}
\label{corollaryA}
The conjecture \ref{conj4} holds for any homogeneous solution of the K\"{a}hler-Ricci flow on a rational homogeneous variety.
\end{corollarynon}

Based on the works of Sesum \cite{Sesum1}, Enders, M\"{u}ller, Topping \cite{EndersMullerTopping}, and Bamler \cite{Bamler}, on the convergence of Ricci-flows with bounded curvature, one can also conclude from Theorem \ref{Theo1} that singularity models of compact simply connected homogeneous K\"{a}hler manifolds are non flat homogeneous gradient shrinking solitons. Under a mild assumption on the scalar curvature of the initial metric, this last fact was also shown in \cite{Optimal} and \cite{Longtimehomogeneous} in the general setting of the homogeneous Ricci flow with finite-time singularity. The key point in the proof of item $(4)$ of our main result is to show that Eq. (\ref{conj1}) holds for any homogeneous solution of the K\"{a}hler-Ricci flow, from this we show that the scalar curvature of such solutions controls the norm of the Ricci curvature tensor. Combining this last fact with \cite[Theorem 4]{Optimal}, we achieve the upper bound for the norm of the Riemann curvature tensor along the flow without any assumption on the scalar curvature of the initial homogeneous K\"{a}hler metric, i.e., we show that in the setting of homogeneous solutions of the K\"{a}hler-Ricci flow the upper bound for the scalar curvature of Eq. (\ref{conj1}) implies the upper bound for the norm of the Riemann curvature tensor as in Eq. (\ref{conj2}). The proof of Conjecture \ref{conj4} follows from item $(6)$ of Theorem \ref{Theo1}, and it is independent of the aforementioned facts. Actually, in order to obtain the uniform upper bound for the diameter of $(X_{P},\omega(t))$ and the uniform lower bound for the first non-zero eigenvalue $\lambda_{1}(t) = \lambda_{1}(X_{P},\omega(t))$ of the Laplacian $\Delta_{\omega(t)} = {\rm{div} \circ {\rm{grad}}}$, we prove that ${\rm{Ric}}(\omega(t)) \geq \frac{1}{C(\omega_{0})}$, for all $t \in [0,T)$, where $C(\omega_{0})$ is the uniform constant given in Eq. (\ref{lowerboundricci}) depending only on $\omega_{0}$. Then, we apply, respectively, Myers's theorem \cite{MYERS} and Lichnerowicz's theorem \cite{Andre}. The upper bound for $\lambda_{1}(t)$ is obtained combining Bourguignon-Li-Yau estimate \cite{BourguignonLiYau}, see also \cite{Arezzo}, \cite{BiliottiGhigi}, the classical Borel-Weil theorem \cite{Serre}, \cite{BorelHirzenbruch}, and the Weyl dimension formula (e.g. \cite{Humphreys}). It is worth pointing out that, from Theorem \ref{Theo1}, we have a rich source of examples which illustrate the results provided in \cite{BamlerZhang}. As an application of Theorem \ref{Theo1}, we study the relationship between numerical invariants associated to ample divisors and certain numerical invariants arising from homogeneous solutions of the K\"{a}hler-Ricci flow. By considering the isomorphism
\begin{equation}
{\text{Hom}}(P,\mathbbm{C}^{\times}) \cong {\rm{Cl}}(X_{P}),
\end{equation}
we investigate the consequences of Theorem \ref{Theo1} from the point of view of intersection theory (e.g. \cite{Fulton}). In this setting, we have the following corollary:

\begin{corollarynon}
\label{corollary2}
In the previous theorem, if $\omega_{0} \in 2\pi c_{1}(\mathcal{O}(D))$, for some ample divisor $D \in {\rm{Div}}(X_{P})$, then the unique smooth solution $\omega(t)$ of the K\"{a}hler-Ricci flow on $X_{P}$ starting at $\omega_{0}$ also satisfies the following:
\begin{enumerate}
\item[1)] $ \displaystyle{ \omega(t) = \sum_{\alpha \in \Sigma \backslash \Theta} \big ( D_{t} \cdot \mathbbm{P}_{\alpha}^{1}\big)\sqrt{-1} \partial \overline{\partial}\log \big (||s_{U}v_{\varpi_{\alpha}}^{+}||^{2}\big ), \ \ \forall t \in [0,T)}$,
\end{enumerate}
where $(D_{t})_{t \in [0,T)}$ is a family of $\mathbbm{R}$-divisors, such that $\frac{d}{d t}D_{t} = K_{X_{P}}$ and $D_{0} = D$;
\begin{enumerate}
\item[2)] $\displaystyle{T = \mathscr{T}(D) = \frac{1}{\tau(D)}}$, where $\tau(D)$ is the nef value of the line bundle $\mathcal{O}(D) \to X_{P}$;
\item[3)] $\displaystyle{R(t) = -\sum_{\beta \in \Pi^{+} \backslash \langle \Theta \rangle^{+}}\frac{d}{dt}\log \bigg \{ \sum_{\alpha \in \Sigma \backslash \Theta} \big ( D_{t} \cdot \mathbbm{P}_{\alpha}^{1}\big) \langle \varpi_{\alpha}, h_{\beta}^{\vee} \rangle\bigg \}, \ \ \forall t \in [0,T)}$;
\item[4)] For all $0 \leq t < T$ we have
\begin{equation}
(2\pi)^{n}\Big [1-\tau(D)t \Big]^{n}\frac{{\rm{deg}}(D)}{n!} \leq {\rm{Vol}}(X_{P},\omega(t)) \leq (2\pi)^{n}\Big [1-\tau(D)t \Big]\frac{{\rm{deg}}(D)}{n!};
\end{equation}

\item[5)] $\displaystyle{{\rm{Ric}}(\omega(t)) \geq \frac{1}{\mathscr{C}(D)}}$, such that $\displaystyle{\frac{\mathscr{C}(D)}{2} = \max_{\alpha \in \Sigma \backslash \Theta} \frac{(D \cdot  \mathbbm{P}_{\alpha}^{1})}{\langle \delta_{P}, h_{\alpha}^{\vee} \rangle}}$, for all $t \in [0,T)$;
\item[6)] Particularly, the first non-zero eigenvalue $\lambda_{1}(X_{P},\omega_{0})$ of the Laplacian $\Delta_{\omega_{0}} = {\rm{div} \circ {\rm{grad}}}$ satisfies 
\begin{equation}
\label{eigenvalueOkoukov}
\frac{2}{\mathscr{C}(D)} \leq \lambda_{1}(X_{P},\omega_{0}) \leq 2n\bigg [ \frac{ \#(\Delta(D) \cap \mathbbm{Z}^{n})}{ \#(\Delta(D) \cap \mathbbm{Z}^{n}) - 1}\bigg ],
\end{equation}
\end{enumerate}
where $\Delta(D)$ is a Newton–Okounkov body associated to $D \in {\rm{Div}}(X_{P})$. Further, $(D_{t})_{t \in [0,T)}$, $\mathscr{T}(D)$ and $\mathscr{C}(D)$ depend only on the numerical equivalence class of $D$.

\end{corollarynon}

The result above shows that the behavior of certain geometric quantities along the K\"{a}hler-Ricci flow associated to a homogeneous initial data $(X_{P},\omega_{0})$, where $\omega_{0} \in 2\pi c_{1}(\mathcal{O}(D))$, for some ample divisor $D \in {\rm{Div}}(X_{P})$, are controlled by the numerical invariants $\mathscr{T}(D)$ and $\mathscr{C}(D)$. In the setting of Eq. (\ref{conj1}), Eq. (\ref{conj2}) and Conjecture \ref{conj4}, we obtain from Corollary \ref{corollary2} that 
\begin{equation}
\label{numeridalcurvature}
R(t) \leq \frac{n}{\mathscr{T}(D) - t}, \ \ \ \ |{\rm{Rm}}| \leq \frac{C(n)}{\mathscr{T}(D) - t}, \ \ \ {\text{and}} \ \ \ {\rm{diam}}(X_{P},\omega(t)) \leq \pi \sqrt{(2n-1)\mathscr{C}(D)},
\end{equation}
for every $t \in [0,\mathscr{T}(D))$. The result of item $(6)$ of the Corollary \ref{corollary2} above provides upper and lower bounds for the first non-zero eigenvalue $\lambda_{1}(X_{P},\omega_{0})$ in terms of the numerical invariants $\mathscr{C}(D)$ and $\Delta(D)$. The Newton–Okounkov body $\Delta(D)$ which appears in Eq. (\ref{eigenvalueOkoukov}) is obtained from the string polytope (e.g. \cite{Littelmann}, \cite{BerensteinZelevinsky}) which parameterises a crystal bases for the irreducible $\mathfrak{g}^{\mathbbm{C}}$-module defined by $H^{0}(X_{P},\mathcal{O}(D))$, see for instance \cite{Kaveh}. As in the toric case \cite{LegendreDias}, item $(6)$ also shows that one can compute explicit upper bounds for the first non-zero eigenvalue associated to integral homogeneous K\"{a}hler metrics in terms of the convex geometry and combinatorics of convex polytopes. Based on the ideas of Corollary \ref{corollary2}, we make some comments and remarks at the end of this paper about how one can relate the numerical invariants $\mathscr{T}(D)$ and $\mathscr{C}(D)$ to certain well-known invariants which appear in the context of algebraic geometry and symplectic geometry, including the global Seshadri constant of ample line bundles (\cite{Demailly}, \cite{Lazarsfeld}), the maximum possible radius of embeddings of symplectic and K\"{a}hler balls (see \cite{McDuffPolterovich}, \cite{Gromov}, \cite{Biran}), and the log canonical threshold of ample $\mathbbm{Q}$-divisors (e.g. \cite[\S 8 - \S 10]{Kollar}, \cite{DemaillyKollar}, \cite{LazarsfeldII}).

{\bf{Organization of the paper.}} This paper is organized as follows: In Section 2, we review some basic known results on K\"{a}hler-Ricci Flow. In Section 3, we introduce some general results on rational homogeneous varieties to be used in the proof of the main results. In Section 4, we prove Theorem \ref{Theo1} and its corollaries. In Section 5, we make some comments and remarks relating the numerical invatirants obtained from Corollary \ref{corollary2} to certain invariants which appear in the context of algebraic geometry and symplectic geometry.

{\bf{Acknowledgements.}} The author would like to thank Professor Lino Grama and Professor Lucas Calixto for very helpful conversations.

\section{Generalities on K\"{a}hler-Ricci flow}

\subsection{K\"{a}hler-Ricci flow} Let $X$ be a $n$-dimensional compact K\"{a}hler manifold and denote by $\mathcal{K}_{X}$ its K\"{a}hler cone, i.e.,
\begin{equation}
\mathcal{K}_{X} = \big \{ [\omega] \in H^{1,1}(X,\mathbbm{R})  \ \big | \ \omega \ {\text{is a K\"{a}hler form}}\big \}.
\end{equation}


If $\omega(t)$ is a solution of the K\"{a}hler-Ricci flow on $X$ stating at some K\"{a}hler metric $\omega_{0}$, with $0 \leq t <T$, $T \leq \infty$, by taking the cohomology class of Eq. (\ref{KRF}) we see that 
\begin{equation}
\frac{\partial}{\partial t}\omega(t) = - 2\pi c_{1}(X) \Longrightarrow [\omega_{0}] - 2\pi tc_{1}(X) = [\omega(t)] \in \mathcal{K}_{X}, \ \forall t \in [0,T). 
\end{equation}
The converse of the above fact is the content of the following theorem proved in \cite{Cao}, \cite{Tsuji1}, \cite{Tsuji}, \cite{TianZhang}.
\begin{theorem}
\label{ExistenceKRF}
Let $(X,\omega_{0})$ be a compact K\"{a}hler manifold of complex dimension $n$. Then the K\"{a}hler-Ricci flow (\ref{KRF}) has a unique smooth solution $\omega(t)$ defined on a maximal interval $[0,T)$, where $T$ is given by
\begin{equation}
T:= \sup \big \{ t > 0 \ \big | \ [\omega_{0}] - 2\pi t c_{1}(X) \in \mathcal{K}_{X}\big \}.
\end{equation}
\end{theorem}

On a compact K\"{a}hler manifold $(X,\omega_{0})$ one can also consider the 1-parameter family of equations:
\begin{equation}
\label{CEq}
\omega(t) = \omega_{0} - t {\rm{Ric}}(\omega(t)),
\end{equation}
notice that in the above equations the K\"{a}hler classes vary according to the linear relation: $[\omega] = [\omega_{0}] - 2\pi t c_{1}(X)$, where $[\omega] \in H^{2}(X,\mathbbm{R}) \cap H^{1,1}(X)$. In this last setting, we have the following result
\begin{theorem}[\cite{LaNaveTian}]
\label{ExistenceCEq}
For any initial K\"{a}hler metric $\omega_{0}$, there is a smooth family of solutions $\omega(t)$ for (\ref{CEq}) on $[0,T)\times X$, such that 
\begin{equation}
T:= \sup \big \{ t > 0 \ \big | \ [\omega_{0}] - 2\pi t c_{1}(X) \in \mathcal{K}_{X}\big \}.
\end{equation}
\end{theorem}

\begin{remark}The continuity equation (\ref{CEq}) can be regarded as an elliptic version of the K\"{a}hler-Ricci flow. Also, notice that the value $T$ of Theorem \ref{ExistenceCEq} coincides with the maximal existence time of Theorem \ref{ExistenceKRF}.
\end{remark}

Given a compact K\"{a}hler manifold $(X,\omega)$, we will denote by $R(\omega) = {\rm{tr}}_{\omega}({\rm{Ric}}(\omega))$ its associated Chern scalar curvature. It is straightforward to see that 
\begin{equation}
\label{Chernscalar}
{\rm{Ric}}(\omega) \wedge \omega^{n-1} = \frac{1}{n}R(\omega) \omega^{n}.
\end{equation}
Also, from the K\"{a}hler condition, we have that $R(\omega) = {\textstyle{\frac{1}{2}}}{\rm{scal}}(\omega)$, where ${\rm{scal}}(\omega)$ denotes the associated Riemannian scalar curvature. In this setting, for the sake of simplicity, we shall refer to $R(\omega)$ just as scalar curvature. From Eq. (\ref{Chernscalar}) above, and considering 
\begin{equation}
{\rm{Vol}}(X,\omega) = \frac{1}{n!}\int_{X}\omega^{n},
\end{equation}
one can prove the following.
\begin{lemma}
\label{volumeflow}
Under the K\"{a}hler-Ricci flow, the volume of $(X,\omega(t))$ changes by 
\begin{equation}
\frac{d}{ d t}{\rm{Vol}}(X,\omega(t)) = - \frac{1}{n!}\int_{X}R(t) \omega(t)^{n},
\end{equation}
where $R(t) = R(\omega(t))$, for all $0 \leq t  < T$.
\end{lemma}
In this work, it will be useful to consider also the following result.
\begin{lemma}
\label{scalaralongflow}
The scalar curvature $R$ of $\omega = \omega(t)$ evolves by
\begin{equation}
\label{derivativescalar0}
\frac{\partial}{\partial t}R = \Delta R + |{\rm{Ric}}|^{2},
\end{equation}
where $|{\rm{Ric}}|^{2} = ||{\rm{Ric}}(\omega)||_{\omega}^{2}$, for all $0 \leq t  < T$.
\end{lemma}

\begin{remark}
From Eq. (\ref{derivativescalar0}), we have 
\begin{equation}
\label{derivativescalar}
\frac{\partial}{\partial t}R = \Delta R + |{\rm{Ric}}^{\circ}|^{2} + \frac{1}{n}R^{2} \geq \Delta R + \frac{1}{n}R^{2},
\end{equation}
where ${\rm{Ric}}^{\circ}$ is the traceless part of the Ricci form, i.e. ${\rm{Ric}}^{\circ} = {\rm{Ric}} - \frac{R}{n}\omega$.
\end{remark}
\section{Generalities on rational homogeneous varieties}

In this section, we review some basic facts about rational homogeneous varieties. From \cite{BorelRemmert}, the study of a rational homogeneous variety reduces to the study of projective algebraic varieties defined by complex flag varieties
\begin{equation}
X_{P} = G^{\mathbbm{C}}/P,
\end{equation}
where $G^{\mathbbm{C}}$ is a connected simply connected complex simple Lie group and $P \subset G^{\mathbbm{C}}$ is a parabolic Lie subgroup. In what follows, we restrict our attention to complex flag varieties. For more details on the subject presented in this section, we suggest \cite{Akhiezer}, \cite{Flagvarieties}, \cite{HumphreysLAG}, \cite{BorelRemmert}.
\subsection{The Picard group of flag varieties}
\label{subsec3.1}
Let $G^{\mathbbm{C}}$ be a connected, simply connected, and complex Lie group with simple Lie algebra $\mathfrak{g}^{\mathbbm{C}}$. By fixing a Cartan subalgebra $\mathfrak{h}$ and a simple root system $\Sigma \subset \mathfrak{h}^{\ast}$, we have a decomposition of $\mathfrak{g}^{\mathbbm{C}}$ given by
\begin{center}
$\mathfrak{g}^{\mathbbm{C}} = \mathfrak{n}^{-} \oplus \mathfrak{h} \oplus \mathfrak{n}^{+}$, 
\end{center}
where $\mathfrak{n}^{-} = \sum_{\alpha \in \Pi^{-}}\mathfrak{g}_{\alpha}$ and $\mathfrak{n}^{+} = \sum_{\alpha \in \Pi^{+}}\mathfrak{g}_{\alpha}$, here we denote by $\Pi = \Pi^{+} \cup \Pi^{-}$ the root system associated to the simple root system $\Sigma = \{\alpha_{1},\ldots,\alpha_{l}\} \subset \mathfrak{h}^{\ast}$. Let us denote by $\kappa$ the Cartan-Killing form of $\mathfrak{g}^{\mathbbm{C}}$. From this, for every  $\alpha \in \Pi^{+}$ we have $h_{\alpha} \in \mathfrak{h}$, such  that $\alpha = \kappa(\cdot,h_{\alpha})$, and we can choose $x_{\alpha} \in \mathfrak{g}_{\alpha}$ and $y_{-\alpha} \in \mathfrak{g}_{-\alpha}$, such that $[x_{\alpha},y_{-\alpha}] = h_{\alpha}$. From these data, we can define a Borel subalgebra by setting $\mathfrak{b} = \mathfrak{h} \oplus \mathfrak{n}^{+}$. Now we consider the following result (see for instance \cite{Flagvarieties}, \cite{HumphreysLAG}):
\begin{theorem}
Any two Borel subgroups are conjugate.
\end{theorem}
From the result above, given a Borel subgroup $B \subset G^{\mathbbm{C}}$, up to conjugation, we can always suppose that $B = \exp(\mathfrak{b})$. In this setting, given a parabolic Lie subgroup $P \subset G^{\mathbbm{C}}$, without loss of generality we can suppose that
\begin{center}
$P  = P_{\Theta}$, \ for some \ $\Theta \subseteq \Sigma$,
\end{center}
where $P_{\Theta} \subset G^{\mathbbm{C}}$ is the parabolic subgroup which integrates the Lie subalgebra 
\begin{center}

$\mathfrak{p}_{\Theta} = \mathfrak{n}^{+} \oplus \mathfrak{h} \oplus \mathfrak{n}(\Theta)^{-}$, \ with \ $\mathfrak{n}(\Theta)^{-} = \displaystyle \sum_{\alpha \in \langle \Theta \rangle^{-}} \mathfrak{g}_{\alpha}$, 

\end{center}
By definition, it is straightforward to show that $P_{\Theta} = N_{G^{\mathbbm{C}}}(\mathfrak{p}_{\Theta})$, where $N_{G^{\mathbbm{C}}}(\mathfrak{p}_{\Theta})$ is its normalizer in  $G^{\mathbbm{C}}$ of $\mathfrak{p}_{\Theta} \subset \mathfrak{g}^{\mathbbm{C}}$. In what follows it will be useful for us to consider the following basic chain of Lie subgroups

\begin{center}

$T^{\mathbbm{C}} \subset B \subset P \subset G^{\mathbbm{C}}$.

\end{center}
For each element in the aforementioned chain of Lie subgroups we have the following characterization: 

\begin{itemize}

\item $T^{\mathbbm{C}} = \exp(\mathfrak{h})$;  \ \ (complex torus)

\item $B = N^{+}T^{\mathbbm{C}}$, where $N^{+} = \exp(\mathfrak{n}^{+})$; \ \ (Borel subgroup)

\item $P = P_{\Theta} = N_{G^{\mathbbm{C}}}(\mathfrak{p}_{\Theta})$, for some $\Theta \subset \Sigma \subset \mathfrak{h}^{\ast}$. \ \ (parabolic subgroup)

\end{itemize}
Now let us recall some basic facts about the representation theory of $\mathfrak{g}^{\mathbbm{C}}$, more details can be found in \cite{Humphreys}. For every $\alpha \in \Sigma$, we can set 
$$h_{\alpha}^{\vee} = \frac{2}{\kappa(h_{\alpha},h_{\alpha})}h_{\alpha}.$$ 
The fundamental weights $\{\varpi_{\alpha} \ | \ \alpha \in \Sigma\} \subset \mathfrak{h}^{\ast}$ of $(\mathfrak{g}^{\mathbbm{C}},\mathfrak{h})$ are defined by requiring that $\varpi_{\alpha}(h_{\beta}^{\vee}) = \delta_{\alpha \beta}$, $\forall \alpha, \beta \in \Sigma$. We denote by 
$$\Lambda^{+} = \bigoplus_{\alpha \in \Sigma}\mathbbm{Z}_{\geq 0}\varpi_{\alpha},$$ 
the set of integral dominant weights of $\mathfrak{g}^{\mathbbm{C}}$. Let $V$ be an arbitrary finite dimensional $\mathfrak{g}^{\mathbbm{C}}$-module. By considering its weight space decomposition
\begin{center}
$\displaystyle{V = \bigoplus_{\mu \in \Pi(V)}V_{\mu}},$ \ \ \ \ 
\end{center}
such that $V_{\mu} = \{v \in V \ | \ h \cdot v = \mu(h)v, \ \forall h \in \mathfrak{h}\} \neq \{0\}$, $\forall \mu \in \Pi(V) \subset \mathfrak{h}^{\ast}$, from the Lie algebra representation theory we have the following facts:
\begin{enumerate}
\item A highest weight vector (of weight $\lambda$) in a $\mathfrak{g}^{\mathbbm{C}}$-module $V$ is a non-zero vector $v_{\lambda}^{+} \in V_{\lambda}$, such that 
\begin{center}
$x \cdot v_{\lambda}^{+} = 0$, \ \ \ \ \ ($\forall x \in \mathfrak{n}^{+}$).
\end{center}
Such a $\lambda \in \Pi(V)$ satisfying the above condition is called highest weight of $V$;
\item $V$ irreducible $\Longrightarrow$ $\exists$ highest weight vector $v_{\lambda}^{+} \in V$ (unique up to non-zero
scalar multiples) for some $\lambda \in \Pi(V)$; 
\item If $\lambda \in \Lambda^{+}$, then there exists a finite dimensional irreducible $\mathfrak{g}^{\mathbbm{C}}$-module $V$ which has $\lambda$ as highest weight. In this case, we denote $V = V(\lambda)$;

\item For all $\lambda \in \Lambda^{+}$, we have $V(\lambda) = \mathfrak{U}(\mathfrak{g}^{\mathbbm{C}}) \cdot v_{\lambda}^{+}$, where $\mathfrak{U}(\mathfrak{g}^{\mathbbm{C}})$ is the universal enveloping algebra of $\mathfrak{g}^{\mathbbm{C}}$;
\item The fundamental representations are defined by $V(\varpi_{\alpha})$, $\alpha \in \Sigma$; 

\item Given $\lambda \in \Lambda^{+}$, such that $\lambda = \sum_{\alpha}n_{\alpha}\varpi_{\alpha}$, we have
\begin{center}
$\displaystyle v_{\lambda}^{+} = \bigotimes_{\alpha \in \Sigma}(v_{\varpi_{\alpha}}^{+})^{\otimes n_{\alpha}}$ \ \ \ and \ \ \ $\displaystyle V(\lambda) = \mathfrak{U}(\mathfrak{g}^{\mathbbm{C}}) \cdot v_{\lambda}^{+} \subset \bigotimes_{\alpha \in \Sigma}V(\varpi_{\alpha})^{\otimes n_{\alpha}};$
\end{center}
\item For all $\lambda \in \Lambda^{+}$, we have the following correspondence of induced irreducible representations
\begin{center}
$\varrho \colon G^{\mathbbm{C}} \to {\rm{GL}}(V(\lambda))$ \ $\Longleftrightarrow$ \ $\varrho_{\ast} \colon \mathfrak{g}^{\mathbbm{C}} \to \mathfrak{gl}(V(\lambda))$,
\end{center}
such that $\varrho(\exp(x)) = \exp(\varrho_{\ast}x)$, $\forall x \in \mathfrak{g}^{\mathbbm{C}}$, notice that $G^{\mathbbm{C}} = \langle \exp(\mathfrak{g}^{\mathbbm{C}}) \rangle$.
\end{enumerate}
In what follows, for any representation $\varrho \colon G^{\mathbbm{C}} \to {\rm{GL}}(V(\lambda))$, for the sake of simplicity, we shall denote $\varrho(g)v = gv$, for all $g \in G^{\mathbbm{C}}$, and all $v \in V(\lambda)$. Let $G \subset G^{\mathbbm{C}}$ be a compact real form for $G^{\mathbbm{C}}$. Given a complex flag variety $X_{P} = G^{\mathbbm{C}}/P$, regarding $X_{P}$ as a homogeneous $G$-space, that is, $X_{P} = G/G\cap P$, the following theorem allows us to describe all $G$-invariant K\"{a}hler structures on $X_{P}$.
\begin{theorem}[Azad-Biswas, \cite{AZAD}]
\label{AZADBISWAS}
Let $\omega \in \Omega^{1,1}(X_{P})^{G}$ be a closed invariant real $(1,1)$-form, then we have

\begin{center}

$\pi^{\ast}\omega = \sqrt{-1}\partial \overline{\partial}\varphi$,

\end{center}
where $\pi \colon G^{\mathbbm{C}} \to X_{P}$, and $\varphi \colon G^{\mathbbm{C}} \to \mathbbm{R}$ is given by 
\begin{center}
$\varphi(g) = \displaystyle \sum_{\alpha \in \Sigma \backslash \Theta}c_{\alpha}\log \big (||gv_{\varpi_{\alpha}}^{+}|| \big )$, \ \ \ \ $(\forall g \in G^\mathbbm{C})$
\end{center}
with $c_{\alpha} \in \mathbbm{R}$, $\forall \alpha \in \Sigma \backslash \Theta$. Conversely, every function $\varphi$ as above defines a closed invariant real $(1,1)$-form $\omega_{\varphi} \in \Omega^{1,1}(X_{P})^{G}$. Moreover, $\omega_{\varphi}$ defines a $G$-invariant K\"{a}hler form on $X_{P}$ if and only if $c_{\alpha} > 0$,  $\forall \alpha \in \Sigma \backslash \Theta$.
\end{theorem}

\begin{remark}
\label{innerproduct}
It is worth pointing out that the norm $|| \cdot ||$ in the last theorem is a norm induced from some fixed $G$-invariant inner product $\langle \cdot, \cdot \rangle_{\alpha}$ on $V(\varpi_{\alpha})$, for every $\alpha \in \Sigma \backslash \Theta$. 
\end{remark}

\begin{remark}
An important consequence of Theorem \ref{AZADBISWAS} is that it allows us to describe the local K\"{a}hler potential for any homogeneous K\"{a}hler metric in a quite concrete way using geometric tools coming from the representation theory of complex semisimple Lie algebras, for some examples of concrete computations we suggest \cite{CorreaGrama}, \cite{Correa}.
\end{remark}

By means of the above theorem we can describe the unique $G$-invariant representative in each integral class in $H^{2}(X_{P},\mathbbm{Z})$. In fact, consider the associated $P$-principal bundle $P \hookrightarrow G^{\mathbbm{C}} \to X_{P}$. By choosing a trivializing open covering $X_{P} = \bigcup_{i \in I}U_{i}$, in terms of $\check{C}$ech cocycles we can write 
\begin{center}
$G^{\mathbbm{C}} = \Big \{(U_{i})_{i \in I}, \psi_{ij} \colon U_{i} \cap U_{j} \to P \Big \}$.
\end{center}
Given a fundamental weight $\varpi_{\alpha} \in \Lambda^{+}$, we consider the induced character $\chi_{\varpi_{\alpha}} \in {\text{Hom}}(T^{\mathbbm{C}},\mathbbm{C}^{\times})$, such that $(d\chi_{\varpi_{\alpha}})_{e} = \varpi_{\alpha}$. From the homomorphism $\chi_{\varpi_{\alpha}} \colon P \to \mathbbm{C}^{\times}$ one can equip $\mathbbm{C}$ with a structure of $P$-space, such that $pz = \chi_{\varpi_{\alpha}}(p)^{-1}z$, $\forall p \in P$, and $\forall z \in \mathbbm{C}$. Denoting by $\mathbbm{C}_{-\varpi_{\alpha}}$ this $P$-space, we can form an associated holomorphic line bundle $\mathscr{O}_{\alpha}(1) = G^{\mathbbm{C}} \times_{P}\mathbbm{C}_{-\varpi_{\alpha}}$, which can be described in terms of $\check{C}$ech cocycles by
\begin{equation}
\label{linecocycle}
\mathscr{O}_{\alpha}(1) = \Big \{(U_{i})_{i \in I},\chi_{\varpi_{\alpha}}^{-1} \circ \psi_{i j} \colon U_{i} \cap U_{j} \to \mathbbm{C}^{\times} \Big \},
\end{equation}
that is, $\mathscr{O}_{\alpha}(1) = \{g_{ij}\} \in \check{H}^{1}(X_{P},\mathcal{O}_{X_{P}}^{\ast})$, such that $g_{ij} = \chi_{\varpi_{\alpha}}^{-1} \circ \psi_{i j}$, for every $i,j \in I$. 
\begin{remark}
\label{parabolicdec}
We observe that, if we have a parabolic Lie subgroup $P \subset G^{\mathbbm{C}}$, such that $P = P_{\Theta}$, the decomposition 
\begin{equation}
P_{\Theta} = \big[P_{\Theta},P_{\Theta} \big]T(\Sigma \backslash \Theta)^{\mathbbm{C}}, \ \  {\text{such that }} \ \ T(\Sigma \backslash \Theta)^{\mathbbm{C}} = \exp \Big \{ \displaystyle \sum_{\alpha \in  \Sigma \backslash \Theta}a_{\alpha}h_{\alpha} \ \Big | \ a_{\alpha} \in \mathbbm{C} \Big \},
\end{equation}
see for instance \cite[Proposition 8]{Akhiezer}, shows us that ${\text{Hom}}(P,\mathbbm{C}^{\times}) = {\text{Hom}}(T(\Sigma \backslash \Theta)^{\mathbbm{C}},\mathbbm{C}^{\times})$. Therefore, if we take $\varpi_{\alpha} \in \Lambda^{+}$, such that $\alpha \in \Theta$, it follows that $\mathscr{O}_{\alpha}(1) = X_{P} \times \mathbbm{C}$, i.e., the associated holomorphic line bundle $\mathscr{O}_{\alpha}(1)$ is trivial.
\end{remark}

Given $\mathscr{O}_{\alpha}(1) \in {\text{Pic}}(X_{P})$, such that $\alpha \in \Sigma \backslash \Theta$, as described above, if we consider an open covering $X_{P} = \bigcup_{i \in I} U_{i}$ which trivializes both $P \hookrightarrow G^{\mathbbm{C}} \to X_{P}$ and $ \mathscr{O}_{\alpha}(1) \to X_{P}$, by taking a collection of local sections $(s_{i})_{i \in I}$, such that $s_{i} \colon U_{i} \to G^{\mathbbm{C}}$, we can define $q_{i} \colon U_{i} \to \mathbbm{R}^{+}$, such that 
\begin{equation}
\label{functionshermitian}
q_{i} =  {\mathrm{e}}^{-2\pi \varphi_{\varpi_{\alpha}} \circ s_{i}} = \frac{1}{||s_{i}v_{\varpi_{\alpha}}^{+}||^{2}},
\end{equation}
for every $i \in I$. Since $s_{j} = s_{i}\psi_{ij}$ on $U_{i} \cap U_{j} \neq \emptyset$, and $pv_{\varpi_{\alpha}}^{+} = \chi_{\varpi_{\alpha}}(p)v_{\varpi_{\alpha}}^{+}$, for every $p \in P$, such that $\alpha \in \Sigma \backslash \Theta$, the collection of functions $(q_{i})_{i \in I}$ satisfy $q_{j} = |\chi_{\varpi_{\alpha}}^{-1} \circ \psi_{ij}|^{2}q_{i}$ on $U_{i} \cap U_{j} \neq \emptyset$. Hence, we obtain a collection of functions $(q_{i})_{i \in I}$ which satisfies on $U_{i} \cap U_{j} \neq \emptyset$ the following relation
\begin{equation}
\label{collectionofequ}
q_{j} = |g_{ij}|^{2}q_{i},
\end{equation}
such that $g_{ij} = \chi_{\varpi_{\alpha}}^{-1} \circ \psi_{i j}$, where $i,j \in I$. From this, we can define a Hermitian structure $H$ on $\mathscr{O}_{\alpha}(1)$ by taking on each trivialization $f_{i} \colon L_{\chi_{\varpi_{\alpha}}} \to U_{i} \times \mathbbm{C}$ a metric defined by
\begin{equation}
\label{hermitian}
H(f_{i}^{-1}(x,v),f_{i}^{-1}(x,w)) = q_{i}(x) v\overline{w},
\end{equation}
for $(x,v),(x,w) \in U_{i} \times \mathbbm{C}$. The Hermitian metric above induces a Chern connection $\nabla = d + \partial \log H$ with curvature $F_{\nabla}$ satisfying (locally)
\begin{equation}
\displaystyle \frac{\sqrt{-1}}{2\pi}F_{\nabla} \Big |_{U_{i}} = \frac{\sqrt{-1}}{2\pi} \partial \overline{\partial}\log \Big ( \big | \big | s_{i}v_{\varpi_{\alpha}}^{+}\big | \big |^{2} \Big).
\end{equation}
Therefore, by considering the $G$-invariant $(1,1)$-form $\Omega_{\alpha} \in \Omega^{1,1}(X_{P})^{G}$, which satisfies $\pi^{\ast}\Omega_{\alpha} = \sqrt{-1}\partial \overline{\partial} \varphi_{\varpi_{\alpha}}$, where $\pi \colon G^{\mathbbm{C}} \to G^{\mathbbm{C}} / P = X_{P}$, and $\varphi_{\varpi_{\alpha}}(g) = \frac{1}{2\pi}\log||gv_{\varpi_{\alpha}}^{+}||^{2}$, $\forall g \in G^{\mathbbm{C}}$, we have 
\begin{equation}
\Omega_{\alpha} |_{U_{i}} = (\pi \circ s_{i})^{\ast}\Omega_{\alpha} = \frac{\sqrt{-1}}{2\pi}F_{\nabla} \Big |_{U_{i}},
\end{equation}
i.e., $c_{1}(\mathscr{O}_{\alpha}(1)) = [ \Omega_{\alpha}]$, $\forall \alpha \in \Sigma \backslash \Theta$. By considering ${\text{Pic}}(X_{P}) = H^{1}(X_{P},\mathcal{O}_{X_{P}}^{\ast})$, from the ideas described above we have the following result.
\begin{proposition}
\label{C8S8.2Sub8.2.3P8.2.6}
Let $X_{P}$ be a complex flag variety associated to some parabolic Lie subgroup $P = P_{\Theta}$. Then, we have
\begin{equation}
\label{picardeq}
{\text{Pic}}(X_{P}) = H^{1,1}(X_{P},\mathbbm{Z}) = H^{2}(X_{P},\mathbbm{Z}) = \displaystyle \bigoplus_{\alpha \in \Sigma \backslash \Theta}\mathbbm{Z}[\Omega_{\alpha} ].
\end{equation}
\end{proposition}
\begin{proof}

Let us sketch the proof. The last equality on the right-hand side of Eq. (\ref{picardeq}) follows from the following facts:

\begin{itemize}

\item[(i)] $\pi_{2}(X_{P}) \cong \pi_{1}(T(\Sigma \backslash \Theta)^{\mathbbm{C}}) = \mathbbm{Z}^{|\Sigma \backslash \Theta|}$, where $T(\Sigma \backslash \Theta)^{\mathbbm{C}}$ is given as in Remark \ref{parabolicdec};

\item[(ii)] Since $X_{P}$ is simply connected, it follows that $H_{2}(X_{P},\mathbbm{Z}) \cong \pi_{2}(X_{P})$ (Hurewicz's theorem);

\item[(iii)] By taking $\mathbbm{P}_{\alpha}^{1} \hookrightarrow X_{P}$, such that 
\begin{equation}
\label{Scurve}
\mathbbm{P}_{\alpha}^{1} = \overline{\exp(\mathfrak{g}_{-\alpha})x_{0}} \subset X_{P},
\end{equation}
for all $\alpha \in \Sigma \backslash \Theta$, where $x_{0} = eP \in X_{P}$, it follows that 

\begin{center}

$\big \langle c_{1}(\mathscr{O}_{\alpha}(1)), [ \mathbbm{P}_{\beta}^{1}] \big \rangle = \displaystyle \int_{\mathbbm{P}_{\beta}^{1}} c_{1}(\mathscr{O}_{\alpha}(1)) = \delta_{\alpha \beta},$

\end{center}
for every $\alpha,\beta \in \Sigma \backslash \Theta$. Hence, we obtain
\begin{center}

$\pi_{2}(X_{P}) = \displaystyle \bigoplus_{\alpha \in \Sigma \backslash \Theta} \mathbbm{Z} [ \mathbbm{P}_{\alpha}^{1}],$ \ \ and \ \ $H^{2}(X_{P},\mathbbm{Z}) = \displaystyle \bigoplus_{\alpha \in \Sigma \backslash \Theta}  \mathbbm{Z} c_{1}(\mathscr{O}_{\alpha}(1))$.

\end{center}
\end{itemize}
Moreover, form above we also have $H^{1,1}(X_{P},\mathbbm{Z}) = H^{2}(X_{P},\mathbbm{Z})$. In order to conclude the proof, from the Lefschetz theorem on (1,1)-classes \cite{DANIEL}, and from the fact that ${\text{rk}}({\text{Pic}}^{0}(X_{P})) = 0$, we obtain the first equality in Eq. (\ref{picardeq}).
\end{proof}

\begin{remark}[Harmonic 2-forms on $X_{P}$]Given any $G$-invariant Riemannian metric $g$ on $X_{P}$, denoting by $\mathscr{H}^{2}(X_{P},g)$ the space of real harmonic 2-forms on $X_{P}$ with respect to $g$, and by $\mathscr{I}_{G}^{1,1}(X_{P})$ the space of closed invariant real $(1,1)$-forms. Combining the result of Proposition \ref{C8S8.2Sub8.2.3P8.2.6} with \cite[Lemma 3.1]{Takeuchi}, we obtain 
\begin{equation}
\mathscr{I}_{G}^{1,1}(X_{P}) = \mathscr{H}^{2}(X_{P},g). 
\end{equation}
Therefore, the closed $G$-invariant real $(1,1)$-forms described in Theorem \ref{AZADBISWAS} are harmonic with respect to any $G$-invariant Riemannian metric on $X_{P}$.
\end{remark}

\begin{remark}[K\"{a}hler cone of $X_{P}$]It follows from Eq. (\ref{picardeq}) and Theorem \ref{AZADBISWAS} that the K\"{a}hler cone of a complex flag variety $X_{P}$ is given explicitly by
\begin{equation}
\mathcal{K}_{X_{P}} = \displaystyle \bigoplus_{\alpha \in \Sigma \backslash \Theta} \mathbbm{R}^{+}[ \Omega_{\alpha}].
\end{equation}
\end{remark}
\begin{remark}[Projective embedding of $X_{P}$] 
\label{BorelWeil}
From Proposition \ref{C8S8.2Sub8.2.3P8.2.6}, we have the group isomorphism ${\text{Hom}}(P,\mathbbm{C}^{\times}) \cong {\text{Pic}}(X_{P})$ described explicitly by
\begin{equation}
\chi \mapsto  L_{\chi} =  \bigotimes_{\alpha \in \Sigma \backslash \Theta}\mathscr{O}_{\alpha}(1)^{\otimes \langle \chi,h_{\alpha}^{\vee}\rangle} , \ \ \ \ L  \mapsto \chi_{L} = \prod_{\alpha \in \Sigma \backslash \Theta} \chi_{\varpi_{\alpha}}^{\langle c_{1}(L), [ \mathbbm{P}_{\beta}^{1}] \rangle},
\end{equation}
for all $\chi \in {\text{Hom}}(P,\mathbbm{C}^{\times})$ and for all $L \in {\text{Pic}}(X_{P})$, where $\langle \chi,h_{\alpha}^{\vee}\rangle = \langle (d\chi)_{e},h_{\alpha}^{\vee} \rangle$, $\forall \alpha \in \Sigma \backslash \Theta$. For the sake of simplicity, we shall denote  $\mathscr{O}_{\alpha}(1)^{\otimes k} = \mathscr{O}_{\alpha}(k)$, for every $k \in \mathbbm{Z}$, and every $\alpha \in \Sigma \backslash \Theta$. Given $L_{\chi} \in {\text{Pic}}(X_{P})$ we have the following equivalences (e.g. \cite{Serre})
\begin{center}
$L_{\chi}$ is ample $\Longleftrightarrow$ is very ample $\Longleftrightarrow$ $\langle \chi,h_{\alpha}^{\vee} \rangle \in \mathbbm{Z}^{+}$, $\forall \alpha \in \Sigma \backslash \Theta.$
\end{center}
Moreover, for every very ample line bundle $L_{\chi} \in {\text{Pic}}(X_{P})$ we have that $H^{0}(X_{P},L_{\chi}) \cong V(\chi)^{\ast}$ (Borel-Weil, \cite{Serre}, \cite{BorelHirzenbruch}), where $V(\chi)$ is the finite dimensional irreducible $\mathfrak{g}^{\mathbbm{C}}$-module associated to the integral dominant weight $(d\chi)_{e} \in \Lambda^{+}$. Following \cite[Theorem 24.10]{BorelHirzenbruch}, \cite[Example 18.13]{Timashev}, given an ample line bundle $L_{\chi} \in {\text{Pic}}(X_{P})$, we have the degree of the associated projective embedding $X_{P} \hookrightarrow \mathbbm{P}(H^{0}(X_{P},L_{\chi})^{\ast})$ given by
\begin{equation}
\label{degreeformula}
{\rm{deg}}(X_{P},L_{\chi}) := \int_{X_{P}}c_{1}(L_{\chi})^{n} = n! \prod_{\alpha \in \Pi^{+} \backslash \langle \Theta \rangle^{+}}\frac{ \langle \chi,h_{\alpha}^{\vee} \rangle}{ \langle \varrho^{+},h_{\alpha}^{\vee} \rangle},
\end{equation}
\label{Weylformula}
where $\varrho^{+}$ is the half sum of all positive roots and  $n = \dim_{\mathbbm{C}}(X_{P})$. Further, from Weyl dimension formula (e.g. \cite{Humphreys}), in the above setting we have
\begin{equation}
\dim_{\mathbbm{C}}(H^{0}(X_{P},L_{\chi})^{\ast}) = \dim_{\mathbbm{C}}(V(\lambda)) = \prod_{\alpha \succ 0} \frac{\langle (d\chi)_{e} + \varrho^{+},h_{\alpha}\rangle}{ \langle \varrho^{+}, h_{\alpha} \rangle},
\end{equation}
here we consider the partial order: $\alpha \succ \beta$ iff $\alpha - \beta$ is a sum of positive roots.
\end{remark}
\subsection{The first Chern class of flag varieties} In this subsection, we will review some basic facts related to the Ricci form of $G$-invariant K\"{a}hler metrics on flag varieties. 

Let $X_{P}$ be a complex flag manifold associated to some parabolic Lie subgroup $P = P_{\Theta} \subset G^{\mathbbm{C}}$. By considering the identification $T_{x_{0}}^{1,0}X_{P} \cong \mathfrak{m} \subset \mathfrak{g}^{\mathbbm{C}}$, such that 

\begin{center}
$\mathfrak{m} = \displaystyle \sum_{\alpha \in \Pi^{+} \backslash \langle \Theta \rangle^{+}} \mathfrak{g}_{-\alpha}$,
\end{center}
where $x_{0} = eP \in X_{P}$, we have $T^{1,0}X_{P}$ as being a holomoprphic vector bundle, associated to the $P$-principal bundle $P \hookrightarrow G^{\mathbbm{C}} \to X_{P}$, given by

\begin{center}

$T^{1,0}X_{P} = G^{\mathbbm{C}} \times_{P} \mathfrak{m}$.

\end{center}
The twisted product on the right-hand side above is obtained from the isotropy representation ${\rm{Ad}} \colon P \to {\rm{GL}}(\mathfrak{m})$. From this, a straightforward computation shows us that 
\begin{equation}
\label{canonicalbundleflag}
K_{X_{P}}^{-1} = \det \big(T^{1,0}X_{P} \big) =\det \big ( G^{\mathbbm{C}} \times_{P} \mathfrak{m} \big )= L_{\chi_{\delta_{P}}},
\end{equation}
where $\det({\rm{Ad}}(g)) = \chi_{\delta_{P}}^{-1}(g)$, $\forall g \in P$, so $\det \circ {\rm{Ad}} = \chi_{\delta_{P}}^{-1}$. Hence, from the previous results we have 
\begin{equation}
\label{charactercanonical}
\chi_{\delta_{P}} = \displaystyle \prod_{\alpha \in \Sigma \backslash \Theta} \chi_{\varpi_{\alpha}}^{\langle \delta_{P},h_{\alpha}^{\vee} \rangle} \Longrightarrow \det \big(T^{1,0}X_{P} \big) = \bigotimes_{\alpha \in \Sigma \backslash \Theta}\mathscr{O}_{\alpha}(\ell_{\alpha}),
\end{equation}
such that $\ell_{\alpha} = \langle \delta_{P}, h_{\alpha}^{\vee} \rangle, \forall \alpha \in \Sigma \backslash \Theta$. If we consider the invariant K\"{a}hler metric $\rho_{0} \in \Omega^{1,1}(X_{P})^{G}$, locally describe by
\begin{equation}
\label{riccinorm}
\rho_{0}|_{U} = \sum_{\alpha \in \Sigma \backslash \Theta}\langle \delta_{P}, h_{\alpha}^{\vee} \rangle \sqrt{-1} \partial \overline{\partial}\log \big (||s_{U}v_{\varpi_{\alpha}}^{+}||^{2}\big ),
\end{equation}
for some local section $s_{U} \colon U \subset X_{P} \to G^{\mathbbm{C}}$. It is straightforward to see that 
\begin{equation}
\label{ChernFlag}
c_{1}(X_{P}) = \Big [ \frac{\rho_{0}}{2\pi}\Big],
\end{equation}
and by the uniqueness of $G$-invariant representative of $c_{1}(X_{P})$, it follows that 
\begin{center}
${\rm{Ric}}(\rho_{0}) = \rho_{0}$, 
\end{center}
i.e., $\rho_{0} \in \Omega^{1,1}(X_{P})^{G}$ defines a $G$-ivariant K\"{a}hler-Einstein metric (cf. \cite{MATSUSHIMA}). 
\begin{remark}
\label{scalarcurvature}
From the uniqueness of the $G$-invariant representative for $c_{1}(X_{P})$, given any $G$-invariant K\"{a}hler metric $\omega_{\varphi}$, we have that ${\rm{Ric}}(\omega_{\varphi}) = \rho_{0}$. Therefore, the scalar curvature $R(\omega_{\varphi})$ of $\omega_{\varphi}$ is given by
\begin{equation}
R(\omega_{\varphi}) = {\rm{tr}}_{\omega_{\varphi}}({\rm{Ric}}(\omega_{\varphi})) = {\rm{tr}}_{\omega_{\varphi}}(\rho_{0}).
\end{equation}
Since $\rho_{0}$ is harmonic with respect to any $G$-invariant K\"{a}hler metric, we have that $R(\omega_{\varphi})$ is constant. 
\end{remark}

By means of Eq. (\ref{degreeformula}), we can compute the volume of $X_{P}$ with respect to $\rho_{0}$ as follows
\begin{equation}
{\rm{Vol}}(X_{P},\rho_{0}) = \frac{1}{n!} \int_{X_{P}} \rho_{0}^{n} = \frac{(2\pi)^{n}}{n!}{\rm{deg}}(X_{P},K_{X_{P}}^{-1}) = (2\pi)^{n}\prod_{\beta \in \Pi^{+} \backslash \langle \Theta \rangle^{+}}\frac{ \langle \delta_{P},h_{\beta}^{\vee} \rangle}{ \langle \varrho^{+},h_{\beta}^{\vee} \rangle}.
\end{equation}
Since for every $G$-invariant K\"{a}hler metric $\omega_{\varphi}$ we have ${\rm{Ric}}(\omega_{\varphi}) = {\rm{Ric}}(\rho_{0}) = \rho_{0}$, it follows that $\frac{\det(\omega_{\varphi})}{\det(\rho_{0})}$ is constant, thus
\begin{equation}
{\rm{Vol}}(X_{P},\omega_{\varphi})  = \frac{\det(\omega_{\varphi})}{\det(\rho_{0})}{\rm{Vol}}(X_{P},\rho_{0}). 
\end{equation}
Denoting $V_{0} = {\rm{Vol}}(X_{P},\rho_{0})$, and computing the (constant) value of $\frac{\det(\omega_{\varphi})}{\det(\rho_{0})}$ at $x_{0} = eP \in X_{P}$, we have the following result.
\begin{theorem}[Azad-Biswas, \cite{AZAD}]
\label{volumeflag}
The volume of $X_{P}$ with respect to an arbitrary $G$-invariant K\"{a}hler metric $\omega_{\varphi}$, induced by some
\begin{center}
$\varphi(g) = \displaystyle \sum_{\alpha \in \Sigma \backslash \Theta}c_{\alpha}\log \big (||gv_{\varpi_{\alpha}}^{+}|| \big )$, \ \ \ \ $(\forall g \in G^\mathbbm{C})$
\end{center}
such that $c_{\alpha} > 0$, $\forall \alpha \in \Sigma \backslash \Theta$, is given by 
\begin{equation}
{\rm{Vol}}(X_{P},\omega_{\varphi}) = V_{0}\frac{\prod_{\beta \in \Pi^{+} \backslash \langle \Theta \rangle^{+}}\Big [\sum_{\alpha \in \Sigma \backslash \Theta}c_{\alpha} \langle \varpi_{\alpha}, h_{\beta}^{\vee} \rangle \Big ] }{\prod_{\beta \in \Pi^{+} \backslash \langle \Theta \rangle^{+}} \Big [\sum_{\alpha \in \Sigma \backslash \Theta}\langle \delta_{P}, h_{\alpha}^{\vee} \rangle \langle \varpi_{\alpha}, h_{\beta}^{\vee} \rangle \Big ]}.
\end{equation}
\end{theorem}

\begin{remark}
In order to perform some local computations we shall consider the open set $U^{-}(P) \subset X_{P}$ defined by the ``opposite" big cell in $X_{P}$. This open set is a distinguished coordinate neighbourhood $U^{-}(P) \subset X_{P}$ of $x_{0} = eP \in X_{P}$ defined as follows

\begin{equation}
\label{bigcell}
 U^{-}(P) =  B^{-}x_{0} = R_{u}(P_{\Theta})^{-}x_{0} \subset X_{P},  
\end{equation}
 where $B^{-} = \exp(\mathfrak{h} \oplus \mathfrak{n}^{-})$, and
 
 \begin{center}
 
 $R_{u}(P_{\Theta})^{-} = \displaystyle \prod_{\alpha \in \Pi^{-} \backslash \langle \Theta \rangle^{-}}N_{\alpha}^{-}$, \ \ (opposite unipotent radical)
 
 \end{center}
with $N_{\alpha}^{-} = \exp(\mathfrak{g}_{\alpha})$, $\forall \alpha \in \Pi^{-} \backslash \langle \Theta \rangle^{-}$. It is worth mentioning that the opposite big cell defines a contractible open dense subset in $X_{P}$, thus the restriction of any vector bundle over this open set is trivial. For further results we suggest \cite{MONOMIAL}.
\end{remark}

\begin{proposition}
\label{lowerricc}
Let $\omega_{\varphi}$ be a $G$-invariant K\"{a}hler metric on $X_{P}$ induced by
\begin{center}
$\varphi(g) = \displaystyle \sum_{\alpha \in \Sigma \backslash \Theta}c_{\alpha}\log \big (||gv_{\varpi_{\alpha}}^{+}|| \big )$, \ \ \ \ $(\forall g \in G^\mathbbm{C})$
\end{center}
such that $c_{\alpha} > 0$, $\forall \alpha \in \Sigma \backslash \Theta$. Then, for all $x \in X_{P}$ and all $v \in T_{x}X_{P}$, such that $\omega_{\varphi}(v,Jv) = 1$, the following holds
\begin{equation}
\label{lowericcibound}
{\rm{Ric}}(\omega_{\varphi})(v,Jv) \geq \min_{\alpha \in \Sigma \backslash \Theta}  \frac{\langle \delta_{P}, h_{\alpha}^{\vee} \rangle }{c_{\alpha}}. 
\end{equation}
\end{proposition}
\begin{proof}
Since ${\rm{Ric}}(\omega_{\varphi}) = \rho_{0}$ is $G$-invariant, it suffices to check Eq. (\ref{lowericcibound}) at the point $x_{0} = eP \in X_{P}$. To this aim, let ${\mathcal{H}}_{\varphi}$ and ${\mathcal{H}}_{\rho_{0}}$ be the Hermitian structures induced on the holomorphic tangent bundle $T^{1,0}X_{P}$, respectively, by $\omega_{\varphi}$ and $\rho_{0}$, that is, 
\begin{center}
${\mathcal{H}}_{\varphi}(Y,Z) = -\sqrt{-1}\omega_{\varphi}(Y,\overline{Z})$ \ \ \ and \ \ \ ${\mathcal{H}}_{\rho_{0}}(Y,Z) = -\sqrt{-1}\rho_{0}(Y,\overline{Z})$,
\end{center}
for all $Y,Z \in T^{1,0}X_{P}$. A straightforward computation shows that 
\begin{center}
$\omega_{\varphi}(v,Jv) = {\mathcal{H}}_{\varphi}\big (\frac{1}{2}(v - \sqrt{-1}Jv),\frac{1}{2}(v - \sqrt{-1}Jv) \big )$ \ \ \ and \ \ \ $\rho_{0}(v,Jv) = {\mathcal{H}}_{\rho_{0}}\big (\frac{1}{2}(v - \sqrt{-1}Jv),\frac{1}{2}(v - \sqrt{-1}Jv) \big ),$
\end{center}
for all $\forall v \in TX_{P}$. From above, it follows that 
\begin{center}
${\rm{Ric}}(\omega_{\varphi})(v,Jv) = {\mathcal{H}}_{\rho_{0}}\big (\frac{1}{2}(v - \sqrt{-1}Jv),\frac{1}{2}(v - \sqrt{-1}Jv) \big ), \ \ \ \forall v \in TX_{P}$.
\end{center}
By considering the coordinate neighborhood $U^{-}(P) \subset X_{P}$ of $x_{0} \in X_{P}$ defined by the opposite big cell (see Eq. \ref{bigcell}), we obtain a suitable basis for $T_{x_{0}}^{1,0}X_{P}$, given by $Y^{\ast}_{\beta} = \frac{\partial}{\partial z}|_{z = 0}\exp(zy_{\beta})x_{0}$, $\beta \in \Pi^{-} \backslash \langle \Theta \rangle^{-}$. The vectors $Y_{\beta}^{\ast}$, $\beta \in \Pi^{-} \backslash \langle \Theta \rangle^{-}$, are orthogonal relative to any $(T^{\mathbbm{C}} \cap G)$-invariant Hermitian form. Moreover, we have 
\begin{equation}
{\mathcal{H}}_{\varphi}(Y_{\beta}^{\ast},Y_{\beta}^{\ast}) = \sum_{\alpha \in \Sigma \backslash \Theta } \frac{c_{\alpha}}{2} \langle \varpi_{\alpha},h_{\beta}^{\vee} \rangle \ \ \ {\text{and}} \ \ \ {\mathcal{H}}_{\rho_{0}}(Y_{\beta}^{\ast},Y_{\beta}^{\ast}) = \sum_{\alpha \in \Sigma \backslash \Theta } \frac{\langle \delta_{P},h_{\alpha}^{\vee} \rangle}{2} \langle \varpi_{\alpha},h_{\beta}^{\vee} \rangle,
\end{equation}
for every $\beta \in \Pi^{-} \backslash \langle \Theta \rangle^{-}$, see for instance \cite{AZAD}. Hence, from the expression above we obtain
\begin{equation}
{\mathcal{H}}_{\rho_{0}}(Y_{\beta}^{\ast},Y_{\beta}^{\ast}) = \sum_{\alpha \in \Sigma \backslash \Theta }  \frac{\langle \delta_{P},h_{\alpha}^{\vee} \rangle}{c_{\alpha}}\frac{c_{\alpha}}{2} \langle \varpi_{\alpha},h_{\beta}^{\vee} \rangle \geq \min_{\alpha \in \Sigma \backslash \Theta} \bigg \{ \frac{\langle \delta_{P}, h_{\alpha}^{\vee} \rangle }{c_{\alpha}}\bigg \} {\mathcal{H}}_{\varphi}(Y_{\beta}^{\ast},Y_{\beta}^{\ast}), 
\end{equation}
for all $\beta \in \Pi^{-} \backslash \langle \Theta \rangle^{-}$. Therefore, combining the above facts, we obtain
\begin{equation}
{\rm{Ric}}(\omega_{\varphi})(v,Jv)  \geq \min_{\alpha \in \Sigma \backslash \Theta} \bigg \{ \frac{\langle \delta_{P}, h_{\alpha}^{\vee} \rangle }{c_{\alpha}}\bigg \}\omega_{\varphi}(v,Jv), \ \ \ \ \forall v \in T_{x_{0}}X_{P}.
\end{equation}
By taking $v \in T_{x_{0}}X_{P}$, such that $\omega_{\varphi}(v,Jv) = 1$, we obtain the inequality (\ref{lowericcibound}) at $x_{0} = eP \in X_{P}$. From the $G$-invariance of $\omega_{\varphi}$ and $\rho_{0}$ we conclude the proof.
\end{proof}

\subsection{Schubert cycles, divisors and line bundles} The aim of this subsection is to recall some general well-known facts on Schubert cycles and their relationship with divisors and line bundles. The details about the facts which we cover in this subsection can be found in \cite{BernsteinGelfand}, \cite{FultonWoodward}, \cite{Brion}, \cite{Popov} see also \cite[\S 17 and \S 18]{Timashev}. 

Following the notation of the previous sections, for every $\alpha \in \Pi^{+}$, consider the root reflection $r_{\alpha} \colon \mathfrak{h}^{\ast} \to \mathfrak{h}^{\ast}$, defined by
\begin{equation}
r_{\alpha}(\phi) = \phi - \langle \phi,h_{\alpha}^{\vee} \rangle \alpha, \ \ \ \ \forall \phi \in \mathfrak{h}^{\ast}.
\end{equation}
From above the Weyl group associated to the root system $\Pi$ is defined by $\mathscr{W} = \langle r_{\alpha} \ | \ \alpha \in \Sigma \rangle$. Under the identification $\mathscr{W} \cong N_{G^{\mathbbm{C}}}(T^{\mathbbm{C}})/T^{\mathbbm{C}}$, by abuse of notation, for any $w \in \mathscr{W}$, we still denote by $w \in G^{\mathbbm{C}}$ one of its representative in $G^{\mathbbm{C}}$. Given a parabolic subgroup $P = P_{\Theta} \subset G^{\mathbbm{C}}$, we denote by $\mathscr{W}_{P}$ the subgroup of $\mathscr{W}$ generated by the reflections $r_{\alpha}$, $\alpha \in \Theta$, and by $\mathscr{W}^{P}$ the quotient $\mathscr{W}/\mathscr{W}_{P}$. Also, we identify $\mathscr{W}^{P}$ with the set of minimal length representatives in $\mathscr{W}$. By considering the $B$-orbit $Bwx_{0} \subset X_{P}$ (Bruhat cell), for every $w \in \mathscr{W}^{P}$, we have a cellular decomposition for $X_{P}$ given by
\begin{equation}
X_{P} = \coprod_{w \in \mathscr{W}^{P}}Bwx_{0}, \ \ \ ({\text{Bruhat decomposition}})
\end{equation}
In the above decomposition we have $Bwx_{0} \cong \mathbbm{C}^{\ell(w)}$, for every $w \in \mathscr{W}^{P}$, where $\ell(w)$ is the length\footnote{$\ell(w)$ denotes the length of a reduced (i.e. minimal) decomposition of $w$ as a product of simple reflections, e.g. \cite{Humphreys}.} of $w \in \mathscr{W}^{P}$. The Schubert varieties are defined by the closure of the above cells; we denote them by $X_{P}(w) = \overline{Bwx_{0}}$, $\forall w \in \mathscr{W}^{P}$. Notice that $\mathbbm{P}_{\alpha}^{1} = X_{P}(r_{\alpha})$, $\forall \alpha \in \Sigma \backslash \Theta$, and it is straightforward to show that the Mori cone ${\rm{NE}}(X_{P})$ is generated by the rational curves $[\mathbbm{P}_{\alpha}^{1}] \in \pi_{2}(X_{P})$, $\forall \alpha \in \Sigma \backslash \Theta$. Similarly, we let $Y_{P}(w) = \overline{B^{-}wx_{0}}$ be the opposite Schubert variety associated to $w \in \mathscr{W}^{P}$; it is a variety of codimension $\ell(w)$, and denoting by $w_{0} \in \mathscr{W}$ the element of maximal length, it follows that $Y_{P}(w) = w_{0}X_{P}(w_{0}w)$, for all $w \in \mathscr{W}^{P}$. For the sake of simplicity, we shall denote $w^{\vee} = w_{0}w$, for all $w \in \mathscr{W}^{P}$. The irreducible $B$-stable divisors of $X_{P}$ are the Schubert varieties of codimension 1 (Schubert divisors). We shall denote them by
\begin{equation}
D_{\alpha} = X_{P}(r_{\alpha}^{\vee}) = w_{0}Y_{P}(r_{\alpha}), \ \ \ \ \forall \alpha \in \Sigma \backslash \Theta.
\end{equation}
Under the map $\mathcal{O} \colon {\rm{Div}}(X_{P}) \to {\rm{Pic}}(X_{P})$, $D \mapsto \mathcal{O}(D)$, we have $\mathcal{O}(D_{\alpha}) = \mathscr{O}_{\alpha}(1)$, $\forall \alpha \in \Sigma \backslash \Theta$. Also, considering the divisor class group\footnote{The symbol ``$\sim$" stands for linear equivalence. Notice that, since $H^{2}(X_{P},\mathbbm{Z})$ is torsion-free, from Lefschetz theorem on $(1,1)$-classes we have that numerically equivalent divisors are in fact linearly equivalent, see for instance \cite{Lazarsfeld}.} ${\rm{Cl}}(X_{P}) = {\rm{Div}}(X_{P})/\sim$, it follows that 
\begin{equation}
{\rm{Cl}}(X_{P}) = \bigoplus_{\alpha \in \Sigma \backslash \Theta}\mathbbm{Z}[D_{\alpha}];
\end{equation}
\begin{remark}
By means of the above results, given $[D] \in {\rm{Cl}}(X_{P})$, we have $D \sim \sum_{\alpha \in \Sigma \backslash \Theta}(D \cdot \mathbbm{P}_{\alpha}^{1})D_{\alpha}$, where $(D \cdot \mathbbm{P}_{\alpha}^{1}) := [D] \cdot [\mathbbm{P}_{\alpha}^{1}], \forall \alpha \in  \Sigma \backslash \Theta$. Thus, we obtain a group isomorphism ${\text{Hom}}(P,\mathbbm{C}^{\times}) \cong {\rm{Cl}}(X_{P})$, such that
\begin{equation}
\label{characterdivisor}
\chi \mapsto  [D_{\chi}] =  \sum_{\alpha \in \Sigma \backslash \Theta}\langle \chi,h_{\alpha}^{\vee}\rangle[D_{\alpha}] , \ \ \ \ [D]  \mapsto \chi_{D} = \prod_{\alpha \in \Sigma \backslash \Theta} \chi_{\varpi_{\alpha}}^{(D \cdot \mathbbm{P}_{\alpha}^{1})},
\end{equation}
for all $\chi \in {\text{Hom}}(P,\mathbbm{C}^{\times})$ and for all $[D] \in {\rm{Cl}}(X_{P})$, where $\langle \chi,h_{\alpha}^{\vee}\rangle = \langle (d\chi)_{e},h_{\alpha}^{\vee} \rangle$, $\forall \alpha \in \Sigma \backslash \Theta$. Under the identification $ {\rm{Pic}}(X_{P}) \cong {\text{Hom}}(P,\mathbbm{C}^{\times}) \cong {\rm{Cl}}(X_{P})$, for the sake of simplicity, we shall denote the canonical line bundle and the canonical divisor of $X_{P}$ just by $K_{X_{P}}$. From above, we have 
\begin{equation}
\label{canonicalclass}
K_{X_{P}} = -\sum_{\alpha \in \Sigma \backslash \Theta}\langle \delta_{P},h_{\alpha}^{\vee}\rangle D_{\alpha}.
\end{equation}
\end{remark}

It will be important for us to consider the following invariant.
\begin{definition}
Let $X$ be a projective variety whose canonical bundle $K_{X}$ is not nef and let $L \in {\rm{Pic}}(X)$ be an ample line bundle. The nef value $\tau(X,L)$ of $L$ is defined as
\begin{equation}
\tau(X,L) = \inf \bigg \{ \frac{p}{q} \in \mathbbm{Q} \ \ \Big | \ \ K_{X}^{\otimes p} \otimes L^{\otimes q} \ \ {\text{is nef}} \bigg \}.
\end{equation}
\end{definition}
In the particular case that $X = X_{P}$, the next result provides a concrete description for the nef value of every ample line bundle $L \in {\rm{Pic}}(X_{P})$.
\begin{theorem}[\cite{Snow}]
Given an ample line bundle $L \in {\rm{Pic}}(X_{P})$, we have
\begin{equation}
\tau(X_{P},L)  = \max_{\alpha \in \Sigma \backslash \Theta} \frac{\langle \delta_{P}, h_{\alpha}^{\vee} \rangle}{\langle \chi_{L},h_{\alpha}^{\vee}\rangle},
\end{equation}
where $\chi_{L} \colon P \to \mathbbm{C}^{\times}$ is the character associated to $L$ by the isomorphism ${\text{Pic}}(X_{P}) \cong {\text{Hom}}(P,\mathbbm{C}^{\times})$.
\end{theorem}

\begin{remark}
For every ample divisor $D \in {\rm{Div}}(X_{P})$, we shall denote $\tau(D):=\tau(X_{P},\mathcal{O}(D))$.
\end{remark}

\subsection{Newton–Okounkov bodies and string polytopes} In this subsection, we review some basic facts and generalities about Newton–Okounkov bodies and string polytopes associated to flag varieties.

Given an ample divisor $D \in {\rm{Pic}}(X_{P})$, let
\begin{equation}
R(X_{P},D) := \bigoplus_{n \geq 0}H^{0}(X_{P},\mathcal{O}(nD)),
\end{equation}
denote the associated ring of global sections. By fixing some total order $\leq$ on $\mathbbm{Z}^{n}$, where $n = \dim_{\mathbbm{C}}(X_{P})$, we have the following definition (e.g. \cite{KavehKhovanski}).
\begin{definition}
A map ${\rm{v}} \colon R(X_{P},D) \backslash \{0\} \to \mathbbm{Z}^{n}$ is called a valuation if for all $c \in \mathbbm{C}^{\times}$, $f,g \in R(X_{P},D) \backslash \{0\}$ the following holds:
\begin{enumerate}
\item[(i)] ${\rm{v}}(cf) = {\rm{v}}(f)$;
\item[(ii)] ${\rm{v}}(fg) = {\rm{v}}(f) + {\rm{v}}(g)$;
\item[(iii)] ${\rm{v}}(f+g) \geq \min\{{\rm{v}}(f),{\rm{v}}(g)\}$ (if $f + g \neq 0$).
\end{enumerate}
We say that ${\rm{v}}$ has full rank if $\dim_{\mathbbm{R}}(\langle Im({\rm{v}}) \rangle_{\mathbbm{R}}) = n$.
\end{definition}

\begin{definition}
Given a valuation ${\rm{v}}$ we define the valuation semigroup with respect to $(X_{P},D)$  as being the graded semigroup $\Gamma_{{\rm{v}}}(D) \subset \mathbbm{N} \times \mathbbm{Z}^{n}$ given by
\begin{equation}
\Gamma_{{\rm{v}}}(D):= \Big \{ (m,{\rm{v}}(f)) \ \Big | \ 0 \neq f \in H^{0}(X_{P},\mathcal{O}(mD)), \ m > 0\Big \} \subset \mathbbm{N} \times \mathbbm{Z}^{n}.
\end{equation}
\end{definition}

In the above setting, we denote by
\begin{equation}
\mathcal{C}(\Gamma_{{\rm{v}}}(D)) := \overline{{\text{cone}}(\Gamma_{{\rm{v}}}(D))} \subset \mathbbm{R}^{n+1},
\end{equation}
the closed convex cone (with vertex at the origin) spanned by $\Gamma_{{\rm{v}}}(D)$, i.e., the intersection of all the closed convex cones containing $\Gamma_{{\rm{v}}}(D)$. From this we have the following definition.
\begin{definition}[Newton–Okounkov body, \cite{Okounkov1}, \cite{Okounkov1}, \cite{KavehKhovanski}] The Newton–Okounkov body $\Delta_{{\rm{v}}}(D)$ associated to a valuation semigroup $\Gamma_{{\rm{v}}}(D)$ is defined by the slice of the cone $\mathcal{C}(\Gamma_{{\rm{v}}}(D))$ at $m = 1$ projected to $\mathbbm{R}^{n}$, via the projection on the second factor $(m,a) \to a$. In other words,
\begin{equation}
\Delta_{{\rm{v}}}(D) = {\text{closed convex hull of}} \bigcup_{m \geq 1} \frac{1}{m}\Big \{ {\rm{v}}(f) \ \Big | \ 0 \neq f \in H^{0}(X_{P},\mathcal{O}(mD))\Big \} \subset \mathbbm{R}^{n}.
\end{equation}
\end{definition}
\begin{remark}
In general, the convex body $\Delta_{{\rm{v}}}(D)$ is not necessarily a polytope, and, as we have seen, its construction depends on the choice of ${\rm{v}}$. As we shall see bellow, under a suitable choice of ${\rm{v}}$, we can attach to every ample divisor $D \in {\rm{Div}}(X_{P})$ a Newton–Okounkov body which is in fact a rational convex polytope (i.e., with rational vertices) satisfying some interesting properties. 
\end{remark}
In \cite{Littelmann} and \cite{BerensteinZelevinsky}, the authors construct a remarkable parameterization, called the string parameterization, for the elements of a crystal basis by the integral points in certain polytopes. These polytopes are known as string polytopes and their construction depends on the choice of a reduced decomposition\footnote{For every $w \in \mathscr{W}$, $\underline{w} = (r_{\alpha_{1}}, \ldots, r_{\alpha_{k}})$ stands for a reduced decomposition $w = r_{\alpha_{1}} \cdots r_{\alpha_{k}}$  ($\ell(w) = k$).} $\underline{w_{0}}$ for the longest element $w_{0} \in \mathscr{W}$. More precisely, fixed a reduced decomposition $\underline{w_{0}}$, there is a rational polyhedral cone $\mathcal{C}_{\underline{w}_{0}}$ in $\Lambda_{\mathbbm{R}}^{+} \times \mathbbm{R}^{N}$, where $\Lambda_{\mathbbm{R}}^{+}$ is the positive Weyl chamber and $N = \ell(w_{0}) = \#(\Pi^{+})$. From this, the string polytope $\Delta_{\underline{w_{0}}}(\lambda)$ of $\lambda \in \Lambda^{+}$ is defined by
\begin{equation}
\Delta_{\underline{w_{0}}}(\lambda) = \Big \{ a \in \mathbbm{R}^{N} \ \Big | \ (\lambda,a) \in  \mathcal{C}_{\underline{w}_{0}} \Big \} \subset \mathbbm{R}^{N}.
\end{equation}
In other words, the string polytope $\Delta_{\underline{w_{0}}}(\lambda)$ is the slice of $\mathcal{C}_{\underline{w}_{0}}$ at $\lambda$. In this setting, given $\lambda \in \Lambda^{+}$, we have the following:
\begin{enumerate}
\item[1)] $\Delta_{\underline{w_{0}}}(\lambda)$ is a rational convex polytope;
\item[2)] $\dim_{\mathbbm{C}}(V(\lambda)) = \#(\Delta_{\underline{w_{0}}}(\lambda) \cap \mathbbm{Z}^{N})$;
\item[3)] For every $k > 1$, we have $\Delta_{\underline{w_{0}}}(k\lambda) = k\Delta_{\underline{w_{0}}}(\lambda)$.
\end{enumerate}
More generally, given any $w \in \mathscr{W}$, by fixing a reduced decomposition $\underline{w}$, we can find $w'\in \mathscr{W}$, satisfying $w_{0} = ww'$, and such that $\underline{w_{0}} = (\underline{w},\underline{w'})$ defines a reduced decomposition, see for instance \cite[p. 16]{Humphreysreflection}. From this, for any $\lambda \in \Lambda^{+}$ we can define the string polytope associated to the pair $(w,\lambda)$ by 
\begin{equation}
\label{stringDemazure}
\Delta_{\underline{w}}(\lambda):= \Delta_{\underline{w_{0}}}(\lambda) \cap (\mathbbm{R}^{\ell(w)} \times \{0\}).
\end{equation}
Observing that $w(\lambda)$ defines a weight (a.k.a. extremal weight \cite{Flagvarieties}) for the $\mathfrak{g}^{\mathbbm{C}}$-module $V(\lambda)$, we have the following definition.
\begin{definition}[\cite{Demazure}, \cite{Kumar}, \cite{BrionKumar}]
Let $w \in \mathscr{W}$ and $\lambda \in \Lambda^{+}$. The Demazure module associated to the pair $(w,\lambda)$ is the $\mathfrak{b}$-module $V_{w}(\lambda) \subseteq V(\lambda)$ defined by 
\begin{equation}
V_{w}(\lambda) := \mathfrak{U}(\mathfrak{b}) \cdot V(\lambda)_{w(\lambda)},
\end{equation}
where $\mathfrak{U}(\mathfrak{b})$ is the enveloping algebra of the Borel subalgebra $\mathfrak{b} \subset \mathfrak{g}^{\mathbbm{C}}$ and $V(\lambda)_{w(\lambda)}$ is the weight space of $V(\lambda)$ with weight $w(\lambda)$. In particular, we have $V_{w_{0}}(\lambda) = V(\lambda)$.
\end{definition}
From above, the rational convex polytope $\Delta_{\underline{w}}(\lambda)$ (Eq. (\ref{stringDemazure})) has the property that the number of integral points in it is equal to the dimension of the Demazure module $V_{w}(\lambda)$, i.e. $\dim_{\mathbbm{C}}(V_{w}(\lambda)) = \#(\Delta_{\underline{w}}(\lambda) \cap \mathbbm{Z}^{\ell(w)})$, see \cite{Littelmann}. Therefore, given any flag variety $X_{P}$, one can associate to every ample divisor $D \in {\rm{Div}}(X_{P})$ a string polytope in the following way. Let $w^{P} \in \mathscr{W}$ be the unique minimal length representative of the class $w_{0}\mathscr{W}_{P} \in \mathscr{W}^{P}$. We have that 
\begin{equation}
\ell(w^{P}) = \dim_{\mathbbm{C}}(X_{P}) = n,
\end{equation}
see for instance \cite{MONOMIAL}. Moreover, there exists a unique $w' \in \mathscr{W}_{P}$, satisfying $w_{0} = w^{P}w'$ (e.g. \cite[\S 1.10]{Humphreysreflection}). From this, by taking the reduced decomposition $\underline{w_{0}} = (\underline{w^{P}},\underline{w'})$, for every ample divisor $D \in {\rm{Div}}(X_{P})$, considering the induced character $\chi_{D} \in {\text{Hom}}(P,\mathbbm{C}^{\times})$, we define its associated string polytope  by
\begin{equation}
\Delta_{\underline{w^{P}}}(D) := \Delta_{\underline{w^{P}}}((d\chi_{D})_{e}).
\end{equation}
We observe that, since $D \in {\rm{Div}}(X_{P})$ is assumed to be ample, from the definition of $\chi_{D}$ (see Eq. (\ref{characterdivisor})), for every $\alpha \in \Sigma$, we have
\begin{center}
$r_{\alpha}((d\chi_{D})_{e}) = (d\chi_{D})_{e} \iff \langle (d\chi_{D})_{e},h_{\alpha}^{\vee} \rangle = 0 \iff r_{\alpha} \in \mathscr{W}_{P}$.
\end{center}
Thus, we have $w_{0}((d\chi_{D})_{e}) = w^{P}w'((d\chi_{D})_{e}) = w^{P}((d\chi_{D})_{e})$. Hence, $V(\chi_{D}) = V_{w^{P}}(\chi_{D})$, so we obtain 
\begin{equation}
\label{dimensionlattice}
\dim_{\mathbbm{C}}(H^{0}(X_{P},\mathcal{O}(D))) = \dim_{\mathbbm{C}}(V(\chi_{D})) =  \#(\Delta_{\underline{w^{P}}}(D) \cap \mathbbm{Z}^{n}).
\end{equation}
The relation between Newton–Okounkov bodies and string polytopes associated to ample divisors $D \in {\rm{Div}}(X_{P})$ is provided by the following theorem:

\begin{theorem}[\cite{Kaveh}]
\label{Okounkovstring}
For every ample divisor $D \in {\rm{Div}}(X_{P})$, there exists a valuation ${\rm{v}}_{\underline{w^{P}}}$, such that the string polytope $\Delta_{\underline{w^{P}}}(D)$ can be identified with the Newton–Okounkov body $\Delta_{{\rm{v}}_{\underline{w^{P}}}}(D)$.
\end{theorem}
\begin{remark}
If $D \sim D'$, then $\Delta_{\underline{w^{P}}}(D) = \Delta_{\underline{w^{P}}}(D')$, i.e., the polytope $\Delta_{\underline{w^{P}}}(D)$ is a numerical invariant. 
\end{remark}

\section{Proof of main results}

In this section, we prove all the results stated in the introduction. For the sake of easy reading, we shall restate each result.

\begin{theorem}[Theorem \ref{Theo1}]
\label{ProofTheo1}
Let $\omega_{0}$ be a $G$-invariant K\"{a}hler metric on a rational homogeneous variety $X_{P}$. Then the unique smooth solution $\omega(t)$ of the K\"{a}hler-Ricci flow on $X_{P}$ starting at $\omega_{0}$ satisfies the following:
\begin{enumerate}
\item[1)] $\omega(t)$ can be described locally in the explicit form
\begin{equation}
\omega(t) = \sum_{\alpha \in \Sigma \backslash \Theta}\bigg [ \int_{\mathbbm{P}_{\alpha}^{1}}\frac{\omega_{0}}{2\pi}- t\langle \delta_{P}, h_{\alpha}^{\vee} \rangle \bigg ]\sqrt{-1} \partial \overline{\partial}\log \big (||s_{U}v_{\varpi_{\alpha}}^{+}||^{2}\big ), \ \ \forall t \in [0,T),
\end{equation}
\end{enumerate}
for some local section $s_{U} \colon U \subset X_{P} \to G^{\mathbbm{C}}$, where $\mathbbm{P}_{\alpha}^{1} \subset X_{P}$, $\alpha \in \Sigma \backslash \Theta$, are generators of ${\rm{NE}}(X_{P})$;
\begin{enumerate}
\item[2)] The maximal existence time $T = T(\omega_{0})$ of $\omega(t)$ is given explicitly by
\begin{equation}
\label{singulartime}
T(\omega_{0}) = \min_{\alpha \in \Sigma \backslash \Theta}  \int_{\mathbbm{P}_{\alpha}^{1}}\frac{\omega_{0}}{2\pi \langle \delta_{P}, h_{\alpha}^{\vee} \rangle};
\end{equation}

\item[3)] The scalar curvature $R(t)$ of $\omega(t)$ has the following explicit form
\begin{equation}
\label{scalarflow}
R(t) = -\sum_{\beta \in \Pi^{+} \backslash \langle \Theta \rangle^{+}}\frac{d}{dt}\log \bigg \{ \sum_{\alpha \in \Sigma \backslash \Theta} \bigg [ \int_{\mathbbm{P}_{\alpha}^{1}}\frac{\omega_{0}}{2\pi}- t\langle \delta_{P}, h_{\alpha}^{\vee} \rangle \bigg ] \langle \varpi_{\alpha}, h_{\beta}^{\vee} \rangle\bigg \}, \ \ \forall t \in [0,T);
\end{equation}

\item[4)] For all $0 \leq t < T$ we have
\begin{equation}
\frac{1}{\sqrt{n}(T-t)}\leq \frac{1}{\sqrt{n}} R(t) \leq |{\rm{Ric}}| \leq R(t) \leq \frac{n}{T - t}, \ \ and \ \ |{\rm{Rm}}| \leq \frac{C(n)}{T - t},
\end{equation}
\end{enumerate}
where $C(n)$ is a uniform constant which depends only on $n = \dim_{\mathbbm{C}}(X_{P})$;
\begin{enumerate}
\item[5)] For all $0 \leq t < T$ we have
\begin{equation}
\bigg [1-\frac{t}{T} \bigg]^{n}{\rm{Vol}}(X_{P},\omega_{0}) \leq {\rm{Vol}}(X_{P},\omega(t)) \leq \bigg [1-\frac{t}{T} \bigg] {\rm{Vol}}(X_{P},\omega_{0});
\end{equation}
\item[6)] For all $0 \leq t < T$ we have ${\rm{Ric}}(\omega(t)) \geq \frac{1}{C(\omega_{0})}$, such that 
\begin{equation}
C(\omega_{0}) =  \max_{\alpha \in \Sigma \backslash \Theta}  \int_{\mathbbm{P}_{\alpha}^{1}}\frac{\omega_{0}}{\pi \langle \delta_{P}, h_{\alpha}^{\vee} \rangle}.
\end{equation}
\end{enumerate}
In particular, for all $0 \leq t < T$, it follows that
\begin{equation}
\label{diameigenvalue}
{\rm{diam}}(X_{P},\omega(t)) \leq \pi \sqrt{(2n-1)C(\omega_{0})} \ \ \ \ \ \  {\text{and}} \ \ \ \ \ \ \ \frac{2}{C(\omega_{0})} \leq \lambda_{1}(t) \leq  2R(t) \Bigg [ \prod_{\alpha \succ 0} \frac{\langle \varrho^{+} + \delta_{P},h_{\alpha} \rangle}{\langle \delta_{P}, h_{\alpha} \rangle}\Bigg ],
\end{equation}
where $\lambda_{1}(t) = \lambda_{1}(X_{P},\omega(t))$ is the first non-zero eigenvalue of the Laplacian $\Delta_{\omega(t)} = {\rm{div} \circ {\rm{grad}}}$, $\forall t \in [0,T)$.

\end{theorem}

\begin{proof}
The item (1) follows from the following facts. Given a $G$-invariant K\"{a}hler metric $\omega_{0}$, from Theorem \ref{AZADBISWAS} we have that $\omega_{0} = \omega_{\varphi}$, for some $\varphi \colon G^{\mathbbm{C}} \to \mathbbm{R}$, such that
\begin{center}
$\varphi(g) = \displaystyle \sum_{\alpha \in \Sigma \backslash \Theta}c_{\alpha}\log \big (||gv_{\varpi_{\alpha}}^{+}|| \big )$, \ \ \ \ $(\forall g \in G^\mathbbm{C})$
\end{center}
with $c_{\alpha} > 0$ for all $\alpha \in \Sigma \backslash \Theta$. Moreover, from Proposition \ref{C8S8.2Sub8.2.3P8.2.6} it follows that 
\begin{equation}
c_{\alpha} =  \int_{\mathbbm{P}_{\alpha}^{1}}\frac{\omega_{0}}{\pi},
\end{equation}
for all $\alpha \in \Sigma \backslash \Theta$. Since ${\rm{Ric}}(\omega) = \rho_{0}$ for every $G$-invariant K\"{a}hler metric $\omega$, it follows from Theorem \ref{ExistenceCEq} that the unique smooth solution $\omega(t)$ defined on the maximal interval $[0,T)$ for the K\"{a}hler-Ricci flow starting at a homogeneous K\"{a}hler metric $\omega_{0}$ is given by $\omega(t) = \omega_{0} - t\rho_{0}$. Thus, from the description for $\rho$ provided by Eq. (\ref{riccinorm}), we have 
\begin{equation}
\label{locdescriptionsolution}
\omega(t) = \sum_{\alpha \in \Sigma \backslash \Theta}\bigg [ \int_{\mathbbm{P}_{\alpha}^{1}}\frac{\omega_{0}}{2\pi}- t\langle \delta_{P}, h_{\alpha}^{\vee} \rangle \bigg ]\sqrt{-1} \partial \overline{\partial}\log \big (||s_{U}v_{\varpi_{\alpha}}^{+}||^{2}\big ), \ \ \forall t \in [0,T),
\end{equation}
for some local section $s_{U} \colon U \subset X_{P} \to G^{\mathbbm{C}}$, so we obtain item (1). In order to prove item (2), we observe that 
\begin{equation}
[\omega(t)] \in \mathcal{K}_{X_{P}} \Longleftrightarrow  \int_{\mathbbm{P}_{\alpha}^{1}}\frac{\omega_{0}}{2\pi} - t\langle \delta_{P}, h_{\alpha}^{\vee} \rangle > 0, \ \ \forall \alpha \in \Sigma \backslash \Theta \Longleftrightarrow t < \int_{\mathbbm{P}_{\alpha}^{1}}\frac{\omega_{0}}{2\pi \langle \delta_{P}, h_{\alpha}^{\vee} \rangle}, \ \ \forall \alpha \in \Sigma \backslash \Theta.
\end{equation}
Therefore, we conclude that $T$ is given explicitly by Eq. (\ref{singulartime}). For the proof of item (3), from Lemma \ref{volumeflow} and Remark \ref{scalarcurvature}, we have
\begin{equation}
\frac{d}{ d t}{\rm{Vol}}(X_{P},\omega(t)) = - \frac{1}{n!}\int_{X}R(t) \omega(t)^{n} = -R(t){\rm{Vol}}(X_{P},\omega(t)), 
\end{equation}
which implies that
\begin{equation}
\label{scalarcurvvolume}
R(t) = - \frac{d}{dt} \log {\rm{Vol}}(X_{P},\omega(t)).
\end{equation}
From Theorem \ref{volumeflag}, we have 
\begin{equation}
-\log {\rm{Vol}}(X_{P},\omega(t)) = -\sum_{\beta \in \Pi^{+} \backslash \langle \Theta \rangle^{+}}\log \bigg \{ \sum_{\alpha \in \Sigma \backslash \Theta} \bigg [ \int_{\mathbbm{P}_{\alpha}^{1}}\frac{\omega_{0}}{2\pi}- t\langle \delta_{P}, h_{\alpha}^{\vee} \rangle \bigg ] \langle \varpi_{\alpha}, h_{\beta}^{\vee} \rangle \bigg \} + \ \text{const.},
\end{equation}
and taking the derivative with respect to $t$ on both sides of the above expression, from Eq. (\ref{scalarcurvvolume}) we obtain item (3). In order to prove item (4), firstly, we will show that 
\begin{equation}
\label{scalarineq} 
\frac{1}{T-t}\leq  R(t) \leq \frac{n}{T - t}, 
\end{equation}
for all $t \in [0,T)$. In fact, for every $\beta \in \Pi^{+} \backslash \langle \Theta \rangle^{+}$, consider the linear function $P_{\beta}(t)$ on $[0,T)$ given by
\begin{equation}
P_{\beta}(t) := \sum_{\alpha \in \Sigma \backslash \Theta} \bigg [ \int_{\mathbbm{P}_{\alpha}^{1}}\frac{\omega_{0}}{2\pi}- t\langle \delta_{P}, h_{\alpha}^{\vee} \rangle \bigg ] \langle \varpi_{\alpha}, h_{\beta}^{\vee} \rangle.
\end{equation}
From above, it follows that 
\begin{equation}
R(t) = -\sum_{\beta \in \Pi^{+} \backslash \langle \Theta \rangle^{+}}\frac{1}{P_{\beta}(t)}\frac{d}{dt}P_{\beta}(t) = \sum_{\beta \in \Pi^{+} \backslash \langle \Theta \rangle^{+}}\frac{a_{\beta}}{P_{\beta}(t)},
\end{equation}
where $a_{\beta} = \sum_{\alpha \in \Sigma \backslash \Theta}\langle \delta_{P}, h_{\alpha}^{\vee} \rangle \langle \varpi_{\alpha}, h_{\beta}^{\vee} \rangle$, for every $\beta \in \Pi^{+} \backslash \langle \Theta \rangle^{+}$. Now we observe that, by definition of $T$, for all $0 \leq t < T$, and every $\alpha \in \Sigma \backslash \Theta$, the following holds
\begin{equation}
\int_{\mathbbm{P}_{\alpha}^{1}}\frac{\omega_{0}}{2\pi \langle \delta_{P}, h_{\alpha}^{\vee} \rangle} - t \geq T - t,
\end{equation}
which implies that 
\begin{equation}
\label{inequality}
P_{\beta}(t) = \sum_{\alpha \in \Sigma \backslash \Theta} \bigg [ \int_{\mathbbm{P}_{\alpha}^{1}}\frac{\omega_{0}}{2\pi \langle \delta_{P}, h_{\alpha}^{\vee} \rangle }- t\bigg ]\langle \delta_{P}, h_{\alpha}^{\vee} \rangle \langle \varpi_{\alpha}, h_{\beta}^{\vee} \rangle\geq a_{\beta}(T - t),
\end{equation}
for all $t \in [0,T)$, and for every $\beta \in \Pi^{+} \backslash \langle \Theta \rangle^{+}$. From above we obtain
\begin{equation}
\label{upperscalar}
R(t) = \sum_{\beta \in \Pi^{+} \backslash \langle \Theta \rangle^{+}}\frac{a_{\beta}}{P_{\beta}(t)} \leq \sum_{\beta \in \Pi^{+} \backslash \langle \Theta \rangle^{+}}\frac{1}{T - t} = \frac{\dim_{\mathbbm{C}}(X_{P})}{T - t}.
\end{equation}
Thus, we obtain the upper bound for $R(t)$ as stated in Eq. (\ref{scalarineq}). In order obtain the desired lower bound for $R(t)$, we observe the following. Denoting by $\gamma\in \Sigma \backslash \Theta$ the simple root which satisfies
\begin{equation}
T = \int_{\mathbbm{P}_{\gamma}^{1}}\frac{\omega_{0}}{2\pi \langle \delta_{P}, h_{\gamma}^{\vee} \rangle},
\end{equation}
we have $P_{\gamma}(t) = (T - t)\langle \delta_{P}, h_{\gamma}^{\vee} \rangle$ and $a_{\gamma} = \langle \delta_{P}, h_{\gamma}^{\vee} \rangle$, i.e., for $\beta = \gamma$ the inequality (\ref{inequality}) becomes a equality. Hence, we obtain
\begin{equation}
\label{lowerscalar}
R(t) = \sum_{\beta \in \Pi^{+} \backslash \langle \Theta \rangle^{+}}\frac{a_{\beta}}{P_{\beta}(t)} \geq \frac{a_{\gamma}}{P_{\gamma}(t)} = \frac{1}{T-t}, 
\end{equation}
for all $t \in [0,T)$. From Eq. (\ref{upperscalar}) and Eq. (\ref{lowerscalar}), we conclude that Eq. (\ref{scalarineq}) holds. From Lemma \ref{scalaralongflow}, since $\Delta R = 0$, we obtain
\begin{equation}
|{\rm{Ric}}|^{2} = \frac{\partial}{\partial t}R(t) = \sum_{\beta \in \Pi^{+} \backslash \langle \Theta \rangle^{+}} \bigg [\frac{a_{\beta}}{P_{\beta}(t)} \bigg ]^{2}.
\end{equation}
Therefore, since $\frac{a_{\beta}}{P_{\beta}(t)} > 0$, for all $t \in [0,T)$, and for every $\beta \in \Pi^{+} \backslash \langle \Theta \rangle^{+}$, we have
\begin{equation}
\label{Ricciupper}
|{\rm{Ric}}| \leq \sum_{\beta \in \Pi^{+} \backslash \langle \Theta \rangle^{+}} \bigg |\frac{a_{\beta}}{P_{\beta}(t)} \bigg| = R(t) \leq \frac{n}{T - t}.
\end{equation}
On the other hand, from Eq. (\ref{derivativescalar}) and Eq. (\ref{scalarineq}), we have
\begin{equation}
|{\rm{Ric}}|^{2} \geq \frac{1}{n}R(t)^{2} \Longrightarrow |{\rm{Ric}}| \geq \frac{1}{\sqrt{n}} R(t) \geq \frac{1}{\sqrt{n}(T-t)}.
\end{equation}
Hence, it follows that 
\begin{equation}
\frac{1}{\sqrt{n}(T-t)}\leq \frac{1}{\sqrt{n}} R(t) \leq |{\rm{Ric}}| \leq R(t) \leq \frac{n}{T - t}, \ \ \ \ (0 \leq t < T).
\end{equation}
In order to conclude the proof of item (4), we just need to observe that $|{\rm{Rm}}| \leq C_{0}(n)|{\rm{Ric}}|$, on $[0,T)$, see for instance \cite[Theorem 4]{Optimal}, where $C_{0}(n)$ depends only on $n = \dim_{\mathbbm{C}}(X_{P})$. Combining this last fact with Eq. (\ref{Ricciupper}) we conclude the proof of item (4). The proof of item (5) follows from the previous facts. Actually, from Eq. (\ref{scalarcurvvolume}) we have
\begin{equation}
{\rm{Vol}}(X_{P},\omega(t)) = {\rm{Vol}}(X_{P},\omega_{0})\rm{e}^{-\int_{0}^{t}R(s)ds}.
\end{equation}
Thus, from Eq. (\ref{scalarineq}) we obtain item (5). The upper bound for the diameter given in Eq. (\ref{diameigenvalue}) of item (6) can be obtained as follows. From Eq. (\ref{locdescriptionsolution}) and Theorem \ref{AZADBISWAS}, it follows that $\omega(t) = \omega_{\varphi(t)}$, such that $\varphi(t) \colon G^{\mathbbm{C}} \to \mathbbm{R}$ is defined by
\begin{equation}
\varphi(t)(g) := \sum_{\alpha \in \Sigma \backslash \Theta}c_{\alpha}(t)\log \big (||gv_{\varpi_{\alpha}}^{+}||\big ), \ \forall g \in G^{\mathbbm{C}}, \ {\text{where}} \ \ c_{\alpha}(t) = 2\bigg [ \int_{\mathbbm{P}_{\alpha}^{1}}\frac{\omega_{0}}{2\pi}- t\langle \delta_{P}, h_{\alpha}^{\vee} \rangle \bigg ], \forall \alpha \in \Sigma \backslash \Theta.
\end{equation}
By observing that $c_{\alpha}(t) \leq c_{\alpha}(0)$, for all $t\in [0,T)$, and for every $\alpha \in \Sigma \backslash \Theta$, from Proposition \ref{lowerricc}, given $t \in [0,T)$, we obtain for all $x \in X_{P}$ and all $v \in T_{x}X_{P}$, such that $\omega(t)(v,Jv) = 1$, the following
\begin{equation}
\label{LBRiccicurvflow}
{\rm{Ric}}(\omega(t))(v,Jv) \geq \min_{\alpha \in \Sigma \backslash \Theta}  \frac{\langle \delta_{P}, h_{\alpha}^{\vee} \rangle }{c_{\alpha}(t)} \geq \min_{\alpha \in \Sigma \backslash \Theta}  \frac{\langle \delta_{P}, h_{\alpha}^{\vee} \rangle }{c_{\alpha}(0)}.
\end{equation}
Since $c_{\alpha}(0) = \int_{\mathbbm{P}_{\alpha}^{1}}\frac{\omega_{0}}{\pi}$, for every $\alpha \in \Sigma \backslash \Theta$, if we define 
\begin{equation}
C(\omega_{0}) =  \max_{\alpha \in \Sigma \backslash \Theta}\frac{c_{\alpha}(0)}{\langle \delta_{P}, h_{\alpha}^{\vee} \rangle} = \max_{\alpha \in \Sigma \backslash \Theta}  \int_{\mathbbm{P}_{\alpha}^{1}}\frac{\omega_{0}}{\pi \langle \delta_{P}, h_{\alpha}^{\vee} \rangle},
\end{equation}
it follows that 
\begin{center}
$\displaystyle{C(\omega_{0}) \geq \frac{c_{\alpha}(0)}{\langle \delta_{P}, h_{\alpha}^{\vee} \rangle }, \ \forall \alpha \in \Sigma \backslash \Theta \Longleftrightarrow  \frac{\langle \delta_{P}, h_{\alpha}^{\vee} \rangle }{c_{\alpha}(0)} \geq \frac{1}{C(\omega_{0}) }, \ \forall \alpha \in \Sigma \backslash \Theta \Longleftrightarrow \frac{1}{C(\omega_{0}) } = \min_{\alpha \in \Sigma \backslash \Theta} \frac{\langle \delta_{P}, h_{\alpha}^{\vee} \rangle }{c_{\alpha}(0)}}.$
\end{center}
Therefore, for all $x \in X_{P}$ and all $v \in T_{x}X_{P}$, such that $\omega(t)(v,Jv) = 1$, from Eq. (\ref{LBRiccicurvflow}) and the last fact above, we have
\begin{equation}
\label{Riccilower}
{\rm{Ric}}(\omega(t))(v,Jv) \geq \frac{1}{C(\omega_{0})}.
\end{equation}
By applying Myers's theorem \cite{MYERS}, we obtain that ${\rm{diam}}(X_{P},\omega(t)) \leq \pi \sqrt{(2n-1)C(\omega_{0})}$, $\forall t \in [0,T)$. In order to conclude the proof, denoting by $\Delta_{\omega(t)} = {\rm{div} \circ {\rm{grad}}}$ the Laplace operator on functions on $(X_{P},\omega(t))$, for all $t \in [0,T)$, since Eq. (\ref{Riccilower}) holds for all $t \in [0,T)$, from Lichnerowicz's theorem \cite{Andre} we obtain that the first non-zero eigenvalue $\lambda_{1}(t)$ of $\Delta_{\omega(t)}$ satisfies the desired inequality 
\begin{equation}
\label{lowereigen}
\frac{2}{C(\omega_{0})} \leq \lambda_{1}(t), 
\end{equation}
for every $t \in [0,T)$. Also, considering the homogeneous (irreducible) very ample line bundle $K_{X_{P}}^{-1} \to X_{P}$, from \cite[Theorem 1.1]{Arezzo} and \cite[Theorem 1.1]{BiliottiGhigi}, it follows that 
\begin{equation}
\lambda_{1}(t) = \lambda_{1}(X_{P},\omega(t)) \leq \frac{4\pi h^{0}(K_{X_{P}}^{-1})}{\big (h^{0}(K_{X_{P}}^{-1}) - 1 \big )} \frac{\big \langle c_{1}(X_{P}) \cup [\omega(t)]^{n-1}, [X_{P}] \big \rangle}{ (n-1)!{\rm{Vol}}(X_{P},\omega(t))},
\end{equation}
where $h^{0}(K_{X_{P}}^{-1}) = \dim_{\mathbbm{C}}(H^{0}(X_{P},K_{X_{P}}^{-1})^{\ast})$. Since $c_{1}(X_{P}) = \big [\frac{{\rm{Ric}}(\omega(t))}{2\pi}\big ]$, for all $t \in [0,T)$, from Eq. (\ref{Chernscalar}) we obtain 
\begin{equation}
\big \langle c_{1}(X_{P}) \cup [\omega(t)]^{n-1}, [X_{P}] \big \rangle =\int_{X_{P}} \frac{{\rm{Ric}}(\omega(t))}{2\pi} \wedge \omega(t)^{n-1} = \frac{R(t)}{2\pi n} \int_{X_{P}}\omega(t)^{n} = \frac{R(t)}{2\pi}(n-1)!{\rm{Vol}}(X_{P},\omega(t))
\end{equation}
Moreover, since $K_{X_{P}}^{-1}= L_{\chi_{\delta_{P}}}$ (see Eq. (\ref{canonicalbundleflag})), from Borel-Weil theorem (see Remark \ref{BorelWeil}) it follows that  $H^{0}(X_{P},K_{X_{P}}^{-1})^{\ast} \cong V(\delta_{P})$. Thus, from Weyl's formula (Eq. (\ref{Weylformula})) and the above facts, we obtain 
\begin{equation}
\label{uppereigen}
\lambda_{1}(t) \leq 2R(t) \bigg [ \frac{\dim_{\mathbbm{C}} (V(\delta_{P}))}{\dim_{\mathbbm{C}} (V(\delta_{P})) - 1}\bigg ] = 2R(t) \Bigg [ \prod_{\alpha \succ 0} \frac{\langle \varrho^{+} + \delta_{P},h_{\alpha} \rangle}{\langle \delta_{P}, h_{\alpha} \rangle}\Bigg ],
\end{equation}
for every $t \in [0,T)$. Combining Eq. (\ref{lowereigen}) with Eq. (\ref{uppereigen}) we conclude the proof. 
\end{proof}

From the result above, we have the following corollary.

\begin{corollary}[Corollary \ref{corollaryA}]
The conjecture \ref{conj4} holds for any homogeneous solution of the K\"{a}hler-Ricci flow on a rational homogeneous variety.
\end{corollary}

\begin{corollary}[Corollary \ref{corollary2}]
\label{corollaryDivisors}
In the previous theorem, if $\omega_{0} \in 2\pi c_{1}(\mathcal{O}(D))$, for some ample divisor $D \in {\rm{Div}}(X_{P})$, then the unique smooth solution $\omega(t)$ of the K\"{a}hler-Ricci flow on $X_{P}$ starting at $\omega_{0}$ also satisfies the following:
\begin{enumerate}
\item[1)] $ \displaystyle{ \omega(t) = \sum_{\alpha \in \Sigma \backslash \Theta} \big ( D_{t} \cdot \mathbbm{P}_{\alpha}^{1}\big)\sqrt{-1} \partial \overline{\partial}\log \big (||s_{U}v_{\varpi_{\alpha}}^{+}||^{2}\big ), \ \ \forall t \in [0,T)}$,
\end{enumerate}
where $(D_{t})_{t \in [0,T)}$ is family of $\mathbbm{R}$-divisors, such that $\frac{d}{d t}D_{t} = K_{X_{P}}$ and $D_{0} = D$;
\begin{enumerate}
\item[2)] $\displaystyle{T = \mathscr{T}(D) = \frac{1}{\tau(D)}}$, where $\tau(D)$ is the nef value of the line bundle $\mathcal{O}(D) \to X_{P}$;
\item[3)] $\displaystyle{R(t) = -\sum_{\beta \in \Pi^{+} \backslash \langle \Theta \rangle^{+}}\frac{d}{dt}\log \bigg \{ \sum_{\alpha \in \Sigma \backslash \Theta} \big ( D_{t} \cdot \mathbbm{P}_{\alpha}^{1}\big) \langle \varpi_{\alpha}, h_{\beta}^{\vee} \rangle\bigg \}, \ \ \forall t \in [0,T)}$;
\item[4)] For all $0 \leq t < T$ we have
\begin{equation}
(2\pi)^{n}\Big [1-\tau(D)t \Big]^{n}\frac{{\rm{deg}}(D)}{n!} \leq {\rm{Vol}}(X_{P},\omega(t)) \leq (2\pi)^{n}\Big [1-\tau(D)t \Big]\frac{{\rm{deg}}(D)}{n!};
\end{equation}

\item[5)] $\displaystyle{{\rm{Ric}}(\omega(t)) \geq \frac{1}{{\mathscr{C}}(D)}}$, such that $\displaystyle{\frac{{\mathscr{C}}(D)}{2} = \max_{\alpha \in \Sigma \backslash \Theta} \frac{(D \cdot  \mathbbm{P}_{\alpha}^{1})}{\langle \delta_{P}, h_{\alpha}^{\vee} \rangle}}$, for all $t \in [0,T)$;
\item[6)] The first non-zero eigenvalue $\lambda_{1}(X_{P},\omega_{0})$ of the Laplacian $\Delta_{\omega_{0}} = {\rm{div} \circ {\rm{grad}}}$ satisfies 
\begin{equation}
\label{eigenvaluelattice}
\frac{2}{\mathscr{C}(D)} \leq \lambda_{1}(X_{P},\omega_{0}) \leq 2n\bigg [ \frac{ \#(\Delta(D) \cap \mathbbm{Z}^{n})}{ \#(\Delta(D) \cap \mathbbm{Z}^{n}) - 1}\bigg ],
\end{equation}
\end{enumerate}
where $\Delta(D)$ is a Newton–Okounkov body associated to $D \in {\rm{Div}}(X_{P})$. Further, $(D_{t})_{t \in [0,T)}$, $\mathscr{T}(D)$ and $\mathscr{C}(D)$ depend only on the numerical equivalence class of $D$.
\end{corollary}

\begin{proof}
Given a $G$-invariant representative $\omega_{0} \in 2 \pi c_{1}(\mathcal{O}(D))$, consider the solution of the K\"{a}hler-Ricci flow $(\omega(t))_{t \in [0,T)}$ starting at $\omega_{0}$ provided by Theorem \ref{ProofTheo1}. In order to prove item (1), item (2), and item (3), we just need to observe that $\int_{\mathbbm{P}_{\alpha}^{1}}\frac{\omega_{0}}{2\pi} = (D \cdot \mathbbm{P}_{\alpha}^{1}) = \langle \chi_{D},h_{\alpha}^{\vee} \rangle$, for every $\alpha \in \Sigma \backslash \Theta$, where $\chi_{D}$ is given as in Eq. (\ref{characterdivisor}). Thus, by taking $D_{t} = D + tK_{X_{P}}$, $t \in [0,T(\omega_{0}))$, it follows that 
\begin{equation}
(D_{t} \cdot \mathbbm{P}_{\alpha}^{1}) = \int_{\mathbbm{P}_{\alpha}^{1}}\frac{\omega_{0}}{2\pi}- t\langle \delta_{P}, h_{\alpha}^{\vee} \rangle,
\end{equation}
for all $t \in [0,T(\omega_{0}))$, see for instance Eq. (\ref{canonicalclass}). Moreover, since 
\begin{equation}
\tau(D) = \max_{\alpha \in \Sigma \backslash \Theta} \frac{\langle \delta_{P}, h_{\alpha}^{\vee} \rangle}{\langle \chi_{D},h_{\alpha}^{\vee}\rangle} = \frac{1}{T(\omega_{0})},
\end{equation}
by setting ${\mathscr{T}}(D) := T(\omega_{0})$, from the previous theorem and the above equations we obtain item (1), item (2), and item (3). In order to prove item (4), we notice that 
\begin{equation}
{\rm{deg}}(D) = {\rm{deg}}(X_{P},\mathcal{O}(D)) := \int_{X_{P}}c_{1}(\mathcal{O}(D))^{n} = \frac{n!}{(2\pi)^{n}}{\rm{Vol}}(X_{P},\omega_{0}),
\end{equation}
from above and from the previous theorem, we have item (4). The proof of item (5) and item (6) follows from the previous theorem, and from following facts. At first, we observe that
\begin{equation}
{\mathscr{C}}(D)=  2\max_{\alpha \in \Sigma \backslash \Theta} \frac{(D \cdot \mathbbm{P}_{\alpha}^{1})}{\langle \delta_{P}, h_{\alpha}^{\vee} \rangle} = \max_{\alpha \in \Sigma \backslash \Theta}  \int_{\mathbbm{P}_{\alpha}^{1}}\frac{\omega_{0}}{\pi \langle \delta_{P}, h_{\alpha}^{\vee} \rangle} = C(\omega_{0}).
\end{equation}
Furthermore, since $\mathcal{O}(D) \to X_{P}$ is a homogeneous (irreducible) very ample line bundle, from \cite[Theorem 1.1]{Arezzo} and \cite[Theorem 1.1]{BiliottiGhigi}, we have 
\begin{equation}
\lambda_{1}(X_{P},\omega_{0}) \leq 2n \frac{ h^{0}(\mathcal{O}(D))}{\big (h^{0}(\mathcal{O}(D)) - 1 \big )},
\end{equation}
where $h^{0}(\mathcal{O}(D)) = \dim_{\mathbbm{C}}(H^{0}(X_{P},\mathcal{O}(D))^{\ast}) = \dim_{\mathbbm{C}}(V(\chi_{D}))$. Hence, by taking $\Delta(D) := \Delta_{\underline{w^{P}}}(D)$ provided by Theorem \ref{Okounkovstring}, from Eq. (\ref{dimensionlattice}) we obtain the upper bound in Eq. (\ref{eigenvaluelattice}). In order to conclude the proof, we observe that, since $H^{2}(X_{P},\mathbbm{Z})$ is torsion-free, we have (by definition) that $(D_{t})_{t \in [0,T)}$, $\mathscr{T}(D)$ and $\mathscr{C}(D)$ depend only on the numerical equivalence class of $D$. 
\end{proof}

\section{Final comments}

In this final section, we make some comments and remarks about how one can relate the numerical invariants $\mathscr{T}(D)$ and $\mathscr{C}(D)$, obtained from Corollary \ref{corollary2}, to certain well-known invariants which appear in some different contexts. 

\begin{itemize}
\item In the particular case that $P$ is a Borel subgroup of $G^{\mathbbm{C}}$, the numerical invariant ${\mathscr{T}}(D)$ which defines the maximal existence time for the solution described in Corollary \ref{corollary2} defines an upper bound for the global Seshadri constant  (\cite{Demailly}, \cite{Lazarsfeld}) as follows: If  $P = B$ is a Borel subgroup of $G^{\mathbbm{C}}$, then for every ample divisor $D \in {\rm{Div}}(X_{B})$, we have
\begin{equation}
\label{Sedhadritime}
\epsilon(\mathcal{O}(D)) \leq 2{\mathscr{T}}(D),
\end{equation}
where $\epsilon(\mathcal{O}(D))$ is the global Seshadri constant of the ample line bundle $\mathcal{O}(D) \to X_{B}$. For explicit description of $\epsilon(\mathcal{O}(D))$, see \cite[Corollary 3.6]{FangLittelmannPabiniak}. 

\item Based on the results of McDuff and Polterovich provided in \cite{McDuffPolterovich}, we have a close relation between Seshadri constants and packing numbers arising from symplectic packing problems (e.g. \cite{Gromov}, \cite{Biran}). In the previous setting, regarding $(X_{B},\omega)$ as a symplectic manifold, for some symplectic form $\omega$, and considering its Gromov width \cite{Gromov}
\begin{equation}
w_{G}(X_{B},\omega) = \sup\{\pi r^{2} \ | \ B\big (0;r \big ) \ \text{can be symplectically embedded in} \ (X_{B},\omega) \}, 
\end{equation}
where $B\big (0;r \big ) \subset \mathbbm{C}^{\dim_{\mathbbm{C}}(X_{B})}$ is the open ball of radius $r$ endowed with the standard symplectic form $\omega_{{\text{std}}}$ induced from $\mathbbm{C}^{\dim_{\mathbbm{C}}(X_{B})}$, we can show the following: Let $D \in {\rm{Div}}(X_{B})$ be an ample divisor and $\omega_{D} \in  c_{1}(\mathcal{O}(D))$ the unique $G$-invariant K\"{a}hler form. If there exists a $C^{\infty}$-embedding 
\begin{equation}
\phi \colon \Big (B\big (0;{\textstyle{\sqrt{\frac{r}{\pi}}}} \big ),\omega_{{\text{std}}} \Big ) \hookrightarrow (X_{B},\omega_{D}),
\end{equation}
for some $r > 0$, such that $\phi^{\ast}(\omega_{D}) = \omega_{{\text{std}}}$ (i.e. $\phi$ is a symplectic embedding), then $r \leq 2 {\mathscr{T}}(D)$. In particular, we have $w_{G}(X_{B},\omega_{D}) \leq 2 {\mathscr{T}}(D)$. These results can be easily obtained combining \cite{FangLittelmannPabiniak} with Eq. (\ref{Sedhadritime}). The relation provided between the numerical invariant ${\mathscr{T}}(D)$ and the Gromov width $w_{G}(X_{B},\omega_{D})$ allows us to describe a constraint for embeddings of symplectic balls in terms of the scalar curvature of $(X_{B},\omega_{D})$. More precisely, if $\phi \colon \big (B\big (0;{\textstyle{\sqrt{\frac{r}{\pi}}}} \big ),\omega_{{\text{std}}} \big ) \hookrightarrow (X_{B},\omega_{D})$ is a symplectic embedding, for some $r > 0$, then 
\begin{equation}
R(\omega_{D}) \leq \frac{2 \pi\dim_{\mathbbm{R}}(X_{B}) }{r}.
\end{equation}
\item Recently, it was shown in \cite{Fleming} that the Seshadri constant determines the maximum possible radius of embeddings of K\"{a}hler balls and vice versa. In this setting, from Eq. (\ref{Sedhadritime}), for any $\omega_{D} \in  c_{1}(\mathcal{O}(D))$, if there exists a holomorphic embedding 
\begin{equation}
\phi \colon \Big (B\big (0;{\textstyle{\sqrt{\frac{r}{\pi}}}} \big ),\omega_{{\text{std}}} \Big ) \hookrightarrow (X_{B},\omega_{D}),
\end{equation}
for some $r > 0$, such that $\phi(0) = eB$ and $\phi^{\ast}(\omega_{D}) = \omega_{{\text{std}}}$ (i.e. $\phi$ is a K\"{a}hler packing), then $r \leq 2 \pi {\mathscr{T}}(D)$.
Notice that, different from Eq. (\ref{Sedhadritime}), in this last case the symplectic form $\omega_{D} \in c_{1}(\mathcal{O}(D))$ does not need to be homogeneous. 

\item The numerical invariant ${\mathscr{C}}(D)$, related to the Ricci curvature appearing in Corollary \ref{corollary2}, can be used to define a lower bound for the log canonical threshold associated to ample $\mathbbm{Q}$-divisors. In fact, observing that every $\mathbbm{Q}$-divisor $D \in {\rm{Div}}(X_{P})_{\mathbbm{Q}}$ is $\mathbbm{Q}$-Cartier, following \cite{Pasquier} and \cite{Smirnov}, for every ample $\mathbbm{Q}$-divisor $D \in {\rm{Div}}(X_{B})_{\mathbbm{Q}}$ and every integer $m \geq 1 $ satisfying $mD \in {\rm{Div}}(X_{B})$, the following holds
\begin{equation}
\label{curvaturesing}
\frac{m}{{\mathscr{C}}(mD)} \leq {\rm{lct}}(D),
\end{equation}
where ${\rm{lct}}(D)$ is the log canonical threshold of $D$, see for instance \cite[\S 8 - \S 10]{Kollar}, \cite{DemaillyKollar}, \cite{LazarsfeldII}. In particular, the pair $(X_{B},D)$ is Kawamata log terminal if and only if the inequality $ {\mathscr{C}}(mD) < m$ holds, and log canonical if the inequality $ {\mathscr{C}}(mD) \leq m$ holds. 
In the particular setting of full flag varieties, the result above provides a geometrical meaning (Corollary \ref{corollary2}, item (5)) for the lower bound of ${\rm{lct}}(D)$ introduced in \cite[Theorem 3.2]{Smirnov}. 

\item In \cite{Nadel1}, following the ideas of Kohn \cite{Kohn} and Siu \cite{Siu}, Nadel introduced the concept of multiplier ideal sheaves as obstructing sheaves for the existence of K\"{a}hler-Einstein metrics of positive scalar curvature on certain complex compact manifolds. This formulation in terms of multiplier ideal sheaves opens up many possibilities for relations with complex and algebraic geometry, see for instance \cite{Demailly2}, \cite{Siu}, \cite{BBJ}, \cite{Braun}, \cite{GuanZhou}, and references therein. In this setting, from Eq. (\ref{curvaturesing}) above, denoting by $\mathcal{J}(D) \subseteq \mathcal{O}_{X_{B}}$ the analytic multiplier ideal sheaf associated to the pair $(X_{B},D)$, we have the following (e.g. \cite{LazarsfeldII}, \cite[\S 3]{Kollar}):
\begin{equation}
(X_{B},D) \ \ {\text{is  Kawamata log terminal}} \ (KLT) \iff  \mathcal{J}(D) = \mathcal{O}_{X_{B}}.
\end{equation}
From Eq. (\ref{curvaturesing}), and the geometrical nature of ${\mathscr{C}}(mD)$, the characterization above shows that the triviality of $\mathcal{J}(D)$ imposes constraints on the Riemannian geometry of $(X_{B},\omega_{0})$, where $\omega_{0}$ is the unique $G$-invariant K\"{a}hler metric in $2\pi c_{1}(\mathcal{O}(mD))$. Being more precise, if $\mathcal{J}(D) = \mathcal{O}_{X_{B}}$, i.e., if $(X_{B},D)$ is KLT, then the homogeneous solution of the K\"{a}hler-Ricci flow $\omega(t)$, $0 \leq t < {\mathscr{T}}(mD)$, starting at $\omega_{0} \in 2\pi c_{1}(\mathcal{O}(mD))$, satisfies ${\rm{Ric}}(\omega(t)) > \frac{1}{m}$, for all $t \in [0,{\mathscr{T}}(mD))$. Therefore, from Myers's theorem \cite{MYERS} and Lichnerowicz's theorem \cite{Andre}, in this last setting we have 
\begin{equation}
\mathcal{J}(D) = \mathcal{O}_{X_{B}} \Longrightarrow {\rm{diam}}(X_{B},\omega(t)) \leq \pi \sqrt{(2n-1)m} \ \ \  {\text{and}} \ \ \ \frac{2}{m} \leq \lambda_{1}(X_{B},\omega(t)),
\end{equation}
$\forall t \in [0,{\mathscr{T}}(mD))$. In particular, we see that the triviality of $\mathcal{J}(D)$ imposes constraints on the diameter and on the first non-zero eigenvalue of the Laplacian $\Delta_{\omega(t)} = {\rm{div} \circ {\rm{grad}}}$. 
\end{itemize}

\end{document}